\documentclass{article}

\usepackage{amsmath}
\usepackage{amsthm}
\usepackage{amssymb}
\usepackage[left=2.5cm,right=2.5cm,top=3cm,bottom=3cm]{geometry}
\usepackage{amscd}
\usepackage{stmaryrd}
\usepackage[normalem]{ulem}
\usepackage{MnSymbol}
\usepackage{bbm}
\usepackage[arrow, curve, matrix]{xy}

\usepackage[utf8]{inputenc}
\usepackage[cyr]{aeguill}
\usepackage{xspace}

\usepackage{tabu, multirow}

\usepackage{listings}
\usepackage{verbatim}

\usepackage{mathbbol}
\usepackage{amssymb}             

\DeclareSymbolFontAlphabet{\mathbb}{AMSb}%
\DeclareSymbolFontAlphabet{\mathbbl}{bbold}

\usepackage[table]{xcolor}
\usepackage{multirow}

\usepackage[pdftex]{graphicx}
\usepackage{enumerate}

\usepackage{array}

\usepackage{hyperref}
\usepackage{varioref}
\usepackage[hyperpageref]{backref}

\usepackage{xcolor}
\hypersetup{
    colorlinks,
    linkcolor={red!60!black},
    citecolor={blue!60!black},
    urlcolor={blue!50!black}
}

\theoremstyle{thm} \newtheorem{thm}{Theorem}
\newtheorem{prop}{Proposition} [section]
\newtheorem{lem}[prop]{Lemma}
\newtheorem{corol}[prop]{Corollary}
\newtheorem{conj}[prop]{Conjecture}
\theoremstyle{definition} 
\theoremstyle{remark} 
\theoremstyle{remark} \newtheorem{rem}[prop]{Remark}
\theoremstyle{definition} \newtheorem{defi}[prop]{Definition}

\newtheorem{innercrit}{Condition}
\newenvironment{crit}[1]
  {\renewcommand\theinnercrit{#1}\innercrit}
  {\endinnercrit}

\newcommand{\quotient}[2]{{\left.\raisebox{-.2em}{$#1$}\middle\backslash\raisebox{.2em}{$#2$}\right.}}
\newcommand{\quotientd}[2]{{\left.\raisebox{.2em}{$#1$}\middle\slash\raisebox{-.2em}{$#2$}\right.}}

\newcommand{\abs}[1]{ {\left| #1 \right| } }
\newcommand{\norm}[1] {\| #1 \| }

\newcommand{\Berg}{{\mathrm{Berg}}}

\makeatletter
\newcommand{\xleftrightarrow}[2][]{\ext@arrow 3359\leftrightarrowfill@{#1}{#2}}
\newcommand{\xdashrightarrow}[2][]{\ext@arrow 0359\rightarrowfill@@{#1}{#2}}
\newcommand{\xdashleftarrow}[2][]{\ext@arrow 3095\leftarrowfill@@{#1}{#2}}
\newcommand{\xdashleftrightarrow}[2][]{\ext@arrow 3359\leftrightarrowfill@@{#1}{#2}}
\def\rightarrowfill@@{\arrowfill@@\relax\relbar\rightarrow}
\def\leftarrowfill@@{\arrowfill@@\leftarrow\relbar\relax}
\def\leftrightarrowfill@@{\arrowfill@@\leftarrow\relbar\rightarrow}
\def\arrowfill@@#1#2#3#4{%
  $\m@th\thickmuskip0mu\medmuskip\thickmuskip\thinmuskip\thickmuskip
   \relax#4#1
   \xleaders\hbox{$#4#2$}\hfill
   #3$%
}
\makeatother

\makeatletter
\def\blfootnote{\xdef\@thefnmark{}\@footnotetext}
\makeatother

\makeatletter
\newcommand{\thickhline}{%
    \noalign {\ifnum 0=`}\fi \hrule height 1pt
    \futurelet \reserved@a \@xhline
}
\newcolumntype{"}{@{\hskip\tabcolsep\vrule width 1pt\hskip\tabcolsep}}
\makeatother

\title{Subvarieties of quotients of bounded symmetric domains}
\author{Beno\^{i}t Cadorel}
\date{}

\begin{document}

\maketitle

\begin{abstract}We present a new criterion for the complex hyperbolicity of a non-compact quotient $X$ of a bounded symmetric domain. For each $p \geq 1$, this criterion gives a precise condition under which the subvarieties $V \subset X$ with $\dim V \geq p$ are of general type, and $X$ is $p$-measure hyperbolic.

Then, we give several applications related to ball quotients, or to the Siegel moduli space of principally polarized abelian varieties. For example, we determine effective levels $l$ for which the moduli spaces of genus $g$ curves with $l$-level structures are of general type.
\end{abstract}

\blfootnote{During the preparation of this work, the author was partially supported by the French ANR project "FOLIAGE", Project ID: ANR-16-CE40-0008, and by the French ANR project "GRACK", Project ID: ANR-15-CE40-0003.}

\section{Introduction}

Let $X$ be a complex projective manifold. Recall that we say that $X$ is of \emph{general type} if its canonical bundle $K_X$ is \emph{big}, i.e. if for $m$ large enough, the evaluation map gives a well defined rational map $X \dashrightarrow \mathbb P H^0(X, K_X^{\otimes m})$, birational on its image. Similarly, we say that a projective variety is of general type if its smooth birational models are of general type.  According to the following conjecture of Green, Griffiths and Lang, the geometry of such a manifold should satisfy many restrictions, which extend beyond the frame of pure algebraic geometry:

\begin{conj}[Green-Griffiths \cite{greengriffiths}, Lang \cite{lang}] Let $X$ be a complex projective manifold. Then $X$ is of general type if and only if there exists a proper algebraic subset $\mathrm{Exc}(X) \subsetneq X$ such that
\begin{enumerate}
\item for any entire holomorphic curve $f : \mathbb C \longrightarrow X$ such that $f(\mathbb C) \not\subset \mathrm{Exc}(X)$, $f$ is constant.
\item for any subvariety $V \subset X$ such that $V \not\subset \mathrm{Exc}(X)$, $V$ is of general type.
\end{enumerate}
\end{conj} 
The two properties above are related to the general notion of \emph{complex hyperbolicity}: a manifold satisfying the first one of these is said to be \emph{Brody hyperbolic modulo $\mathrm{Exc}(X)$}, while the second property is a weak case of \emph{algebraic hyperbolicity} (see \cite{dem12a}). 

The study of this conjecture has been addressed on many families of examples of complex manifolds, and in particular, it has led to an important amount of work in the case of compactifications of quotients of bounded symmetric domains, which are the objects we will consider from now on. 
\medskip

Let $\overline{X} = \overline{\quotient{\Gamma}{\Omega}}$ be a smooth compactification of a quotient of a bounded symmetric domain by a torsion free arithmetic lattice (see the beginning of Section \ref{sectioncomp} for more details). In this setting, we know from the work of Tai and Mumford \cite{AMRT, mum77} that there exists a subgroup of finite index $\Gamma_0 \subset \Gamma$, such that for any other subgroup $\Gamma' \subset \Gamma_0$ of finite index, any smooth compactification $\overline{X'} = \overline{\quotient{\Gamma'}{\Omega}}$ is of general type. Thus, there is an infinite number of such varieties on which it makes sense to study the Green-Griffiths-Lang conjecture.

It has been proved by Rousseau \cite{rou15} that if the index $[\Gamma : \Gamma']$ is high enough, then $\overline{X'}$ is Kobayashi hyperbolic modulo its boundary $D'$, which means that the Kobayashi metric (see \cite{kob76}) separates any two points of $\overline{X'}$ which are not in the boundary.  In particular, this shows that the entire curves on $\overline{X'}$ must belong to $D'$, which was first proved by Nadel \cite{nad89}. 

In the same spirit, Brunebarbe \cite{bru16a} recently showed that if $[\Gamma : \Gamma']$ is high enough, any subvariety of $\overline{X'}$ which is not of general type must be included in the boundary $D'$. While the core of the proof of \cite{rou15} was based on simple metric arguments combined with the Ahlfors-Schwarz lemma to bound the size of holomorphic disks contained in $\overline{X'}$, Brunebarbe's proof uses subtle methods related to the theory of variations of Hodge structures, to show the bigness of the cotangent bundles of subvarieties of $\overline{X'}$ which are not included in the boundary. Then, an application of a theorem of Campana and P\u{a}un \cite{campau15} permits to conclude that these varieties are of general type. 

The main goal of this article is to present a unified treatment of these two theorems, which will permit to generalize \cite{rou15} to the case of $p$-measure hyperbolicity (which is a direct generalization of Kobayashi hyperbolicity), and to give a metric point of view on the main result of \cite{bru16a}, not using the difficult theory of variations of Hodge structure.
\medskip

The method we will present consists in constructing a positively curved metric on the canonical bundle $K_V$ of a subvariety $V \subset \overline{X}$. Then, a theorem of Boucksom \cite{bou02} permits to prove the bigness of the canonical bundle of $K_V$, which shows that $V$ is of general type. This kind of technique was already successfully used in a previous work with Y. Brunebarbe \cite{brucad17}, to study the log-canonical bundle of varieties supporting a variation of Hodge structure. Here, the metric in question will be produced by twisting the metric induced by the Bergman metric on $K_V$ ; this method is similar to the main idea of \cite{rou15}. 

\medskip

Let us move on to the statement of our results. Consider a bounded symmetric domain $\Omega$, and let $\Gamma$ be a torsion free arithmetic lattice in $\mathrm{Aut}(\Omega)$. Let $X = \quotient{\Gamma}{\Omega}$, and let $\overline{X} \overset{\pi}{\longrightarrow} \overline{X}^{BB}$ be a desingularization of the Baily-Borel compactification of the quotient $X = \quotient{\Gamma}{\Omega}$. Let $D = \overline{X} \setminus X$ ; we choose $\pi$ so that $D$ is a simple normal crossing divisor. Then $\overline{X}$ supports the nef and big divisor $\pi^\ast K_{\overline{X}^{BB}}$. When $\Gamma$ is a \emph{neat} lattice, we can use \cite{AMRT} to obtain directly a compactification $\overline{X} = X \sqcup D$, with $D$ having simple normal crossings. Remark that by \cite{AMRT, mum77}, when $\Gamma$ is \emph{neat}, $X$ admits some particular compactifications $\overline{X} = X \sqcup D$ for which $\pi^\ast K_{\overline{X}}^{BB}$ is none other than the log-canonical divisor $K_{\overline{X}} + D$ (see Section \ref{sectioncomp}).

Normalize the Bergman metric on $\Omega$ so that $\mathrm{Ric}(h_{\mathrm{Berg}}) = - \omega_{\Berg}$, and let $- \gamma \in \mathbb Q^\ast_+$ be the maximum of the holomorphic sectional curvature on $\Omega$.

\medskip

Our main result is a criterion for algebraic and transcendental hyperbolicity of $\overline{X}$ in term of the positivity of the line bundle $L_0 = \mathcal O_{\overline{X}} \left( \pi^\ast K_{\overline{X}^{BB}} \right)$.

\begin{thm} \label{thmpartial} There exists a series of constants $0 < \gamma = C_1 \leq C_2 ... \leq C_n = 1$, depending only on $\Omega$, such that the following holds. 

Suppose that for some $p \geq 1$, there exists a rational number $\alpha > \frac{1}{C_p}$ such that the $\mathbb Q$-line bundle $L_\alpha = L_0 \otimes \mathcal O_{\overline{X}} (- \alpha D)$ is effective. Then, 
\begin{enumerate}
\itemsep0em
\item Any subvariety $V \subset \overline{X}$, such that $ V \not\subset \mathbb B(L_\alpha) \cup D$, and $\dim V \geq p$, is of general type.
\item $\overline{X}$ is infinitesimally $p$-measure hyperbolic modulo $\mathbb B(L_\alpha) \cup D$. 
\end{enumerate}
In particular, any generically immersive holomorphic map $f :  \mathbb C^p \longrightarrow \overline{X}$ has its image in $\mathbb B(L_\alpha) \cup D$.
\end{thm}

Here, we denote by $\mathbb B(L) = \bigcap_{m \in \mathbb N} \mathrm{Bs} \left(L^{\otimes m} \right)$ the stable base locus of a $\mathbb Q$-line bundle $L$. The last statement about the hyperbolicity of $X$ modulo $\mathbb B(L_\alpha) \cup D$ is a generalization of a theorem of Nadel \cite{nad89}, which treated only the case of entire curves in $\overline{X}$. However, in \cite{nad89}, the stable base locus $\mathbb B(L_\alpha)$ was mistakenly omitted, as pointed to us by Y. Brunebarbe. 

Remark that in the particular case $p=1$, the first point could also be proved using Guenancia's recent result \cite{Guenancia18} (see Remark \ref{remguenancia}). Indeed, if $L_{\alpha}$ is effective, with $\alpha > \frac{1}{\gamma}$, we can twist the restriction of the Bergman metric to $T_V$ in a similar fashion to \cite{cad16}. This yields a singular metric on $V$, bounded everywhere, with negative sectional curvature, which implies by \cite{Guenancia18} that $V$ is of general type.

However, studying metrics directly on $K_V$, as we propose to do here, seems to be a more direct approach, and, as stated above, it has the advantage of providing more precise information for the subvarieties of higher dimension. Also, we are able to obtain an explicit lower bound on the volume of the canonical bundle of the subvarieties of such a quotient, as asserts our next result.

\begin{thm} \label{thmlincomb} Let $\alpha$ be a positive rational number such that $L_\alpha = L_0 \otimes \mathcal O_{\overline{X}} (- \alpha D)$ is effective. Let $V \subset \overline{X}$ be a subvariety with $\dim V \geq p$ and $V \not\subset D \cup \mathbb B(L_\alpha)$, and let $\widetilde{V} \overset{j}{\longrightarrow} V$ be a resolution of singularities such that $F = j^\ast D$ has normal crossings. Denote by $F_{red}$ the associated reduced divisor.

Then, for any $\lambda \in \; ]0, C_p\; \alpha[ \; \cap \; \mathbb Q$, the $\mathbb Q$-line bundle $\mathcal O_{\widetilde{V}} (K_{\widetilde{V}} + (1 - \lambda) F_{red})$ is big, and we have the following lower bound on its volume :
$$
\mathrm{vol} \left(K_{\widetilde{V}} + (1 - \lambda) F_{red}\right) \geq \left( \frac{C_p - \frac{\lambda}{\alpha}}{2 \pi} \right)^{\dim V} \int_{V} \omega_{\mathrm{Berg}}^{\dim V}.
$$

\end{thm}

\medskip

Using Theorem \ref{thmpartial}, we can give a very simple proof of the following strong hyperbolicity result. Its algebraic part is the main result of \cite{bru16a}, while its transcendental part, concerning the Kobayashi metric, was proved in \cite{rou15}.

\begin{thm} [\cite{bru16a}, \cite{rou15}] \label{thmbru} Let $\lambda = \frac{1}{\gamma} \left( \binom{n+1}{2} + 2 \right)$.
Assume that $\Gamma$ is neat, and let $\overline{X}$ be a toroidal compactification of $X$ in the sense of \cite{AMRT}. Let $\overline{X'} \longrightarrow \overline{X}$, be a ramified cover, étale on the open part, and ramifying at orders higher or equal to $\lambda$ on each boundary component. Then $\overline{X'}$ is Kobayashi hyperbolic modulo its boundary $D'$, and any subvariety $V \subset \overline{X'}$, $V \not\subset D'$, is of general type.
\end{thm}

Later on, once a few more notations are introduced, we will give another version of this theorem, more precise if we consider only subvarieties of dimension higher than some constant (see Theorem \ref{mainthm}). For the time being, let us simply mention another version of Theorem \ref{thmpartial} in the case of the ball, based on the results of Bakker and Tsimerman \cite{baktsi15}.

\begin{thm} \label{thmball} Suppose that $\Omega = \mathbb B^n$, and assume that $\Gamma$ has only unipotent parabolic isometries. Let $\overline{X'} \longrightarrow \overline{X}$ be a ramified cover, étale on the open part, and ramifying at orders higher than some integer $l$ on the boundary. Let $p = \lceil \frac{2 \pi}{l} \rceil - 1$. Then $\overline{X'}$ is infinitesimally $p$-measure hyperbolic modulo $D'$, and any subvariety $V \subset \overline{X'}$ with $\dim V \geq p$ and $V \not\subset D'$, is of general type.
\end{thm}

In particular, taking $l = 1$, we obtain the following striking result.

\begin{corol} \label{corolball}
Let $\overline{X} = \overline{\quotient{\Gamma}{\mathbb B^n}}$ be as in Theorem \ref{thmball}. Any subvariety $V \subset \overline{X}$ such that $\dim V \geq 6$ and $V \not\subset D$ is of general type.
\end{corol}

Note that if the Green-Griffiths-Lang conjecture holds true, this last result would imply the following conjecture :

\begin{conj}
Let $\overline{X} = \overline{\quotient{\Gamma}{\mathbb B^n}}$ be as in Theorem \ref{thmball}, with $n \geq 6$. Then there exists an algebraic subset $\Sigma \subset \overline{X}$, with $\dim \Sigma \leq 5$, such that any non-constant entire curve, or any subvariety which is not of general type, is included in $D \cup \Sigma$.
\end{conj}

In the following, we present a few applications of our previous general results, especially in the case of Siegel modular variety.

\subsection{Application to the Siegel modular variety}

We now turn our attention to the case of Siegel modular variety $A_g(l)$, which parametrizes the principally polarized abelian varieties with an $l$-level structure ; this is always a fine moduli space if $l \geq 3$. The variety $A_g$ is a quotient of the Siegel generalized upper half-space $\mathbb H_g$ of dimension $n = \frac{g(g+1)}{2}$, by the lattice $\Gamma(l) = \mathrm{Ker} \left( \mathrm{Sp}(2g, \mathbb Z) \longrightarrow \mathrm{Sp}(2g, \mathbb Z / l \mathbb Z) \right)$. Our results in this setting will be based on the following explicit computation of the constants $C_p$ for the domain $\mathbb H_g$.
 
\begin{prop} \label{propconstants} For some $g \geq 2$, let $\Omega = \mathbb H_g$. For any $p \in \llbracket 1,  \frac{g(g+1)}{2} \rrbracket$, we define $k$ to be the largest integer such that $\frac{k(k+1)}{2} \leq p - 1$, and we let $r = p - 1 - \frac{k(k+1)}{2}$. Then, the value of $C_p$ can be computed in terms of $k$ and $r$, accordingly to the table given in Figure 1.
\begin{figure}[!h]
\centering
\setlength{\tabcolsep}{5pt}
\begin{tabular}{|c|c|c|c|c|c|}
\hline
 & $g-k = 1$ & $g-k=2$ & $g-k = 3$ & $g-k = 4$ & $g- k \geq 5$ \\
\hline
$r =0$ & \multirow{4}{*}{$\frac{r + 2}{g+1}$} & \multicolumn{4}{c|}{$\frac{2}{g-k} \frac{1}{g+1}$} \\
\cline{1-1}\cline{3-6}
$r =1$ &  & $\frac{23}{16} \frac{1}{g+1}$ & $\frac{11}{12} \frac{1}{g+1}$ & $\frac{21}{32} \frac{1}{g+1}$ & \\
\cline{1-1}\cline{3-5}
$r = 2$ & & $\frac{7}{4} \frac{1}{g+1}$ & \multicolumn{3}{c|}{} \\
\cline{1-1}\cline{3-3}
$r = 3$ & & $\frac{31}{16} \frac{1}{g+1}$ & \multicolumn{3}{c|}{$\frac{2}{g-k-1}\frac{1}{g+1}$} \\
\cline{1-1}\cline{3-3}
$r \geq 4$ & & \multicolumn{4}{c|}{} \\
\hline
\end{tabular}
\caption{Values of $C_p$ for the domain $\mathbb H_g$} \label{tablecpag}
\end{figure}
\end{prop}

\medskip
This computation, joint with a small generalization of Theorem \ref{thmpartial} to the case of singular quotients of bounded symmetric domains, gives the following corollary, which is a slight improvement of a result of Weissauer \cite{weissauer86}.

\begin{corol} \label{corolsubv} Let $g \geq 12$ and let $V \subset \overline{A_g}^{BB}$ a subvariety which is not included in the boundary, nor in the set of elliptic singular points. If $\mathrm{codim} (V) \leq g - 12$, then $V$ is of general type.
\end{corol}

Recall that when $l \geq 3$, the lattice $\Gamma(l)$ is \emph{neat} (see \cite{chai85}), so we can use \cite{AMRT} to obtain a smooth toroidal compactification of the moduli space $A_g(l)$. Adding more level structure, we can formulate statements valid for all subvarieties of $\overline{A_g(l)}$. Applying Theorem \ref{thmpartial}, we obtain the following theorem, which improves a result first stated in \cite{bru16a}:

\begin{thm}  \label{thmaghypfull} 
Let $l > 6 g$, and consider a toroidal compactification $\overline{A_g(l)}$, with boundary $D$. Then, $\overline{A_g(l)}$ is Kobayashi hyperbolic modulo $D$, and any subvariety $V \subset \overline{A_g(l)}$ with $V \not\subset D$, is of general type.
\end{thm}

The claim concerning the Kobayashi metric was first proved in \cite{rou15}. If we allow hyperbolicity modulo a subset other than $D$, we can take a uniform lower bound on $n$, as shown by the next result.

\begin{thm} \label{thmaghypmod} 
Let $l \geq 54$. Then there is a proper algebraic subset $Z \subsetneq \overline{A_g(l)}$, such that $\overline{A_g(l)}$ is Kobayashi hyperbolic modulo $Z\cup D$, and such that any subvariety $V \subset \overline{A_g(l)}$, with $V \not\subset D \cup Z$, is of general type.
\end{thm}

\subsection{Application to the moduli space of curves with $l$-level structure}

We now fix some genus $g \geq 2$, and some integer $l \in \mathbb N$. As explained by Deligne and Mumford \cite{delignemumford}, we can define a coarse moduli space $M_{g,l}$ curves of genus $g$ with an $l$-level structure on their Jacobian variety, which comes with a compactification $\overline{M_{g,l}}$. By the work of Mumford and Harris \cite{MH82}, and Farkas \cite{Farkas09}, we know that $\overline{M_{g,l}}$ is of general type if $g = 22$ or $g \geq 24$. In the contrary, when $g \leq 16$, the manifolds $\overline{M_g}$ are not of general type (see \cite{Farkas09a} for a good survey of all these results). Nevertheless, not much seems to have been explicitly stated about the Kodaira dimension of the manifolds $\overline{M_{g,l}}$.
 
In this situation, we can use the previous results to obtain some information about the level structure needed for $\overline{M_{g,l}}$ to be of general type, for the small values of $g$. Indeed, the period map induces a generically immersive morphism $M_{g,l} \longrightarrow A_{g,l}$, and in this situation, we can use Theorem 1 to show that the  compactification $\overline{M_{g,l}}$ is of general type, for many pairs $(g,l)$.

\begin{corol} \label{corolmgn} Let $g \geq 2$, and let $k$ be the largest integer such that $3g - 4 \geq \frac{k(k+1)}{2}$. Let $l \in \mathbb N$ satisfying the following conditions :
\begin{enumerate}
\item if $g \geq 12$, let $l \geq 6(g-k - 1)$ ; 
\item if $g \leq 11$, let $l \geq l_0$, where $l_0$ is given by the following table.

\begin{figure}[!h]
\centering
\setlength{\tabcolsep}{5pt}
\begin{tabu}{|[1pt]c|[1pt]c|c|c|c|c|c|c|c|c|c|[1pt]}
\tabucline[1pt]{1-11}
$g$ & 2 & 3 & 4 & 5 & 6 &7 & 8 & 9 & 10 & 11   \\
\tabucline[1pt]{1-11}
$l_0$ &37 &49 &49  &37 & 25 & 22 &13 &13 &9 &8  \\
\tabucline[1pt]{1-11}
\end{tabu}
\caption{Lower bounds on the $l$-level structure needed to have $\overline{M_{g,l}}$ of general type}
\centering
\end{figure}
\end{enumerate}
Then $\overline{M_{g,l}}$ is of general type.
\end{corol}

\subsection{Organization of the paper}

This article will be organized as follows. We will first introduce the definitions needed to complete the proofs of Theorem \ref{thmpartial} and \ref{thmlincomb}. We will then prove these results, and explain how we can use them to obtain the previous statements concerning the Kobayashi hyperbolicity or the type of a subvariety of \emph{any} dimension (see Theorem \ref{mainthm}).

We will then move on to the particular examples of the ball and the Siegel half space, for which it is possible to compute explicitly the constants $C_p$ mentioned earlier. Our computation will be followed by a presentation of the results concerning the $p$-measure hyperbolicity and the type of the subvarieties of bounded dimension, for a quotient of one of these two domains.

In order to deal with the case of the variety $\overline{A_g}$, it is actually necessary to consider the case of non-necessarily torsion free lattices. In the last section, we will explain one method to extend the field of application of our results, inspired from \cite{weissauer86}. In particular, we will present a generalization of Theorem \ref{thmpartial} to the singular case, in the spirit of \cite{crt17} (see Corollary \ref{corolsing}). To avoid delaying too much the proof of our main results by introducing the necessary definitions, we have preferred to deal only with torsion free lattices in the first three sections; as we will explain in the last section, we can formulate similar results in the torsion case without much change in the arguments.
\medskip

\begin{bfseries}Acknowledgements.\end{bfseries} I would like to thank Erwan Rousseau for guiding me through the existing literature, and for suggesting me many applications of the main result. Also, I am grateful to Yohan Brunebarbe for several enlightening discussions which helped me a lot to clarify the ideas underlying the present work. Finally, I thank Philippe Eyssidieux and S\'ebastien Boucksom, for their useful comments on an earlier version of this work.

\section{Hyperbolicity results in the torsion-free case}

Before starting the proof of our main result, we must introduce a few notations.

\subsection{Upper bounds on the Ricci curvature of subvarieties}  \label{sectioncurvature}

Let $\Omega \subset \mathbb C^n$ be a bounded symmetric domain. Suppose that the Bergman metric on $\Omega$ in normalized as $\mathrm{Ric}(h_{\mathrm{Berg}}) = - \omega_{\mathrm{Berg}}$, and let $\gamma \in \mathbb Q^{\ast}_+$ be such that the holomorphic section curvature is bounded from above by $-\gamma$.

We will first define properly the constants mentioned in Theorem \ref{thmpartial}. Let us recall the following classical definition.
\medskip

\begin{defi} Let $h$ be an hermitian metric on some complex manifold $V$, with curvature form $\Theta$. The \emph{bisectional curvature} of $h$ is defined to be the following function on $T_V \times T_V$:
\begin{equation} \label{defibisectional}
\forall u \in T_{V, x}, \, \forall v \in T_{V, x},  \quad B(u, v) := \frac{i \Theta (u, \overline{u}, v, \overline{v})}{\norm{u}_h^2 \norm{v}_h^2}.
\end{equation}
\end{defi}

\medskip

Now, let $p \in \llbracket 1, n \rrbracket$, and let $x \in \Omega$ and $v$ be a non-zero tangent vector at $x$. Since the bisectional curvature on $\Omega$ is non-positive, we deduce that for any vector subspace $V$ containing $v$, we have
\begin{align} \nonumber
\mathrm{Tr}^V_{h_{\mathrm{Berg}}} i \left. \Theta(h_{\Berg}) \right|_V^V \cdot(v,\overline{v}) & = i \sum_j \Theta(e_j,\overline{e_j}, v, \overline{v})  \\
  & = \sum_{j} B(e_j, v) \norm{v}^2 \leq -\gamma \norm{v}^2\label{eqrestrcurv}
\end{align}
where $\mathrm{Tr}^V_{h_{\mathrm{Berg}}}$ denotes the trace on $V$ with respect to the Bergman metric. Here, $(e_j)_j$ is any unitary frame of $V$ for the metric $h_{\mathrm{Berg}}$. We can take for example $e_1 = \frac{v}{\norm{v}}$, which gives easily the last inequality.

For each $p$, and for any $(x, v)$ as above, we can now define a constant $C_{x,v, p}$ in the following way :
\begin{equation} \label{defCxvp}
C_{x,v,p} = - \sup_{V \ni v, \mathrm{dim} V = p} \frac{ \mathrm{Tr}_{h_{\mathrm{Berg}}}^V i \left. \Theta(h_{\Berg}) \right|_V^V \cdot(v,\overline{v})}{\norm{v}^2}.
\end{equation}
By \eqref{eqrestrcurv}, for any choice of $(x,v,p)$, we have: $C_{x,v,p} \geq \gamma > 0$.

\begin{defi}
For any $p$, we let
$$
C_{p} = \min_{v \in T_{\Omega,x} \setminus \left\{ 0 \right\}} C_{x,v,p}.
$$
\end{defi}
Since $\Omega$ is a homogeneous space, we see immediately that this definition is independent of $x$. The following property is straightforward.

\begin{prop} \label{ineqprop} We have the inequalities
$$
\gamma \leq C_1 \leq C_2 \leq ... \leq C_n = 1.
$$
\end{prop}

The bisectional curvature decreases on subvarieties, so we can find an upper bound on the Ricci curvature on the subvarieties in terms of the constants $C_p$. 

\begin{prop} \label{propricci}
Let $(Y,0) \subset (\Omega,0)$ be a germ of submanifold of dimension $p$. Then, the restriction of the Bergman metric to $Y$ has its Ricci curvature bounded as
\begin{equation} \label{boundriccicurvature}
\mathrm{Ric} \left(\left. h_{\mathrm{Berg}} \right|_Y \right) \leq - C_p \, j^\ast \omega_{\Berg}.   
\end{equation}
where $j : Y \longrightarrow \Omega$ is the embedding map.
\end{prop}
\begin{proof}
Let $v$ be a tangent vector to $Y$ at $0$, and let $V = T_{Y,0}$. We have, by definition of the Ricci curvature:
\begin{align*}
\mathrm{Ric}\left( \left. h_{\Berg} \right|_Y\right) \cdot (v, \overline{v}) & =  \mathrm{Tr}_{h_{\mathrm{Berg}} |_Y}^V i \left. \Theta(h_{\Berg} |_Y) \right|_V^V \cdot(v, \overline{v}) \\
								  & \leq \mathrm{Tr}_{h_{\mathrm{Berg}}}^V i \left. \Theta(h_{\Berg}) \right|_V^V \cdot(v,\overline{v}) \\  
								  & \leq - C_p \norm{v}^2,
\end{align*}
where the first inequality comes from the fact that the bisectional curvature decreases on submanifolds. The last inequality comes simply from the definition of $C_p$.
\end{proof}

\begin{rem} The inequality \eqref{boundriccicurvature} is optimal: if $Y$ osculates a vector space $V \subset T_{\Omega, 0}$ realizing the upper bound in \eqref{defCxvp}, for some $v \in V$ such that $C_{0,v,p} = C_p$, then we have
$$
\mathrm{Ric} \left(\left. h_{\mathrm{Berg}} \right|_Y \right) (v,\overline{v}) = - C_p \, \norm{v}^2_{\Berg}
$$
\end{rem}

\subsection{Smooth compactifications of $X$} \label{sectioncomp}

In this section, we give the proofs of our main hyperbolicity results, in the torsion free case. Before that, let us give a few properties of the smooth compactifications we will consider.
\medskip

Consider then a torsion free arithmetic lattice $\Gamma \subset \mathrm{Aut}(\Omega)$. By \cite{bb66, satake60}, it is possible to compactify the smooth quotient $X = \quotient{\Gamma}{\Omega}$ into a normal projective variety $\overline{X}^{BB}$. Note that in this situation, $\overline{X}^{BB}$ has an ample canonical bundle $K_{\overline{X}^{BB}}$. Consequently, if $\pi : \overline{X} \longrightarrow \overline{X}^{BB}$ is a desingularization map, the pull-back $ L_0 = \pi^\ast \mathcal O_{\overline{X}^{BB}} \left( K_{\overline{X}^{BB}} \right)$ is \emph{nef} and \emph{big}. In the following, we will choose a desingularization $\pi$ which is biholomorphic over $X$, and such that $D = \overline{X} \setminus X$ is a simple normal crossing divisor.

Moreover, when the lattice $\Gamma$ is \emph{neat}, we can use \cite{AMRT} to produce particular smooth compactifications $\overline{X} = X \sqcup D$, on which $D$ is an SNC divisor, with a natural projection map $\pi : \overline{X} \longrightarrow \overline{X}^{BB}$. In this case, we have a natural identification $L_0 \cong \mathcal O_{\overline{X}} \left(K_{\overline{X}} + D \right)$, by \cite{mum77}.

We now state a result which is a slight variant of a lemma already stated in \cite{crt17}. It permits to control the growth of the norm of sections of powers of $\pi^\ast K_{\overline{X}^{BB}}$ near the boundary, which will be of particular importance in the sequel. Recall that the Bergman metric induces a singular metric on $T_{\overline{X}}$, which we will also denote by $h_{\mathrm{Berg}}$, when no risk of confusion results. The restrictions $\pi^\ast K_{\overline{X}^{BB}}|_{X}$ and $K_{\overline{X}}|_{X}$ are canonically identified, so it makes sense to talk about the norm of sections of $ \mathcal O_{\overline{X}} \left( m \, \pi^\ast K_{\overline{X}^{BB}} \right)$.

\begin{lem}[see \cite{crt17}] \label{loggrowth} Let $m \in \mathbb N$, and let $s \in H^0 \left(\overline{X}, L_0^{\otimes m}  \right)$. Then, for any neighborhood of a point of the boundary, there exists some constants $C > 0$ and $\alpha > 0$ such that,
\begin{equation} \label{upperboundlog}
\norm{s}^2_{(\det h_{\Berg}^\ast)^{m}} \leq C \, \abs{ \, \log \abs{w}^2 \, }^\alpha,
\end{equation}
where $w$ is a local equation for $D$.
\end{lem}
\begin{proof} 
The first step of the proof is to show that we can have a bound of the form \eqref{upperboundlog} on a particular normal compactification $\overline{X_1}$ of $X$. To do this, we need to consider a compactification $\overline{X'}$ of a quotient of $\Omega$ by a sub-lattice of $\Gamma$, constructed as follows.

By \cite{borel1969}, we can find a \emph{neat} normal sublattice of finite index $\Gamma' \subset \Gamma$. We can then use \cite{AMRT} to form a smooth compactification of the quotient $X' = \quotient{\Gamma'}{\Omega}$, which will take the form of a smooth projective variety $\overline{X'} = X' \cup D'$, on which $D'$ is an SNC divisor. Let $G = \quotientd{\Gamma}{\Gamma'}$ ; this finite group acts naturally on $\overline{X'}$.

Now, remark that since $\overline{X}^{BB}$ has normal singularities, $s$ descends to a well defined element 
$$
s_0 \in H^0 \left(\overline{X}^{BB}, \mathcal O_{\overline{X}^{BB}} (m\, K_{\overline{X}^{BB}} ) \right).
$$
The section $s_0$ lifts to a $G$-invariant section $s' \in H^0 \left(\overline{X'}, \mathcal O_{\overline{X'}} (m \, (\pi')^\ast  K_{\overline{X'}^{BB}}) \right)$. Here, we have denoted by $\pi' : \overline{X'} \longrightarrow \overline{X'}^{BB}$ the natural projection map.

Since $\Gamma'$ is neat, \cite{mum77} asserts that $s'$ satisfies a bound of the form \eqref{upperboundlog}, on any small open subset of the manifold $\overline{X'}$. Thus, on any neighborhood of a point of the boundary $D' \subset \overline{X'}$, there are some constants $C' > 0$ and $\beta > 0$ such that 
$$
\norm{s'}^2_{(\det h_{\Berg}^\ast)^{m}} \leq C \left( \log \abs{w'}^2 \right)^{\beta},
$$
where $w'$ is a local equation for the boundary $D'$.

Now, let $\overline{X_1}$ be the normal variety $\quotient{G}{\overline{X'}}$, where $\pi_1 : \overline{X_1} \longrightarrow \quotient{G}{\overline{X'}^{BB}} \cong \overline{X}^{BB}$ is the induced projection map. 

Since the restriction $s'|_{X}$ is $G$-invariant, it descends to a global section $s_1$ of $ \mathcal O_{\overline{X_1}} \left( m \, (\pi_1)^{\ast} K_{\overline{X}^{BB}} \right)$ on $\overline{X_1}$. As announced above, we can now check that $\norm{s_1}$ has the required growth. 
 
If $(w_i)_i$ is a set of local equations for the boundary Weil divisor $D_1 :=  \quotient{G}{D'}$, on some open subset of $\overline{X_1}$, we have an equality of ideal sheaves $(w') = \sqrt{ \sum_i (p^\ast w_i)_i }$, where $p : \overline{X'} \longrightarrow \overline{X_1}$. Thus, we deduce that for some suitable $k \in \mathbb N$, we have locally 
\begin{equation}
\norm{s_1}^2 \leq C \left( k  \sum_i \log \abs{w_i} \right)^\beta. \label{boundnorms1}
\end{equation}
We have proved the required inequality on the birational model $\overline{X_1}$, which completes the first part of the proof.

 To conclude, we just need to relate the right hand side of \eqref{upperboundlog} and \eqref{boundnorms1} on the open part $X$, which embeds in both $\overline{X_1}$ and $\overline{X}$. For this, note that the graph of the birational transformation $\overline{X_1} \dashrightarrow \overline{X}$ can be resolved as in the following diagram.
$$
\xymatrix{
 & Z \ar[dl]^{u} \ar[dr]^{}_{v} & \\
\overline{X_1} & X \ar[l] \ar[r]  &  \overline{X},
}
$$
Let $D_Z = u^\ast (D_1)_{red} = v^\ast D$ ; all these divisors are equal because they are all identified to the reduced divisor underlying schematic pullbacks of the boundary of $\overline{X}^{BB}$ to $Z$. Thus, if $w_z$ is a local equation for this divisor above open subsets of $\overline{X_1}$ and $\overline{X}$ as before, then, on the open set $X$, the functions $\sum_i \log \abs{w_i}$,  $\log \abs{w'}$, and $\log \abs{w_r}$ are pairwise commensurable (we say that $f$ and $g$ are commensurable if $f = O(g)$ and $g = O(f)$). Finally, on $X \subset Z$, we have $u^\ast s_1 = v^\ast s$ and thus $\norm{s_1} \circ u = \norm{s} \circ v$, so these commensurability relations combined with \eqref{boundnorms1} give the required inequality \eqref{upperboundlog}. 
\end{proof}

Before starting the proof of our main result, we introduce two constants that will be useful in applications of Theorem \ref{thmpartial}. Recall that since $K_{\overline{X}^{BB}}$ is ample, the line bundle $L_0 =  \mathcal O_{\overline{X}} \left( \pi^\ast K_{\overline{X}^{BB}} \right)$ is big. This property is open, so for any $\alpha > 0$ small enough, $L_\alpha =  L_0 \otimes \mathcal O(- \alpha D)$ is effective. This allows us to define the following constant:

\begin{defi} $\alpha_{\mathrm{eff}} = \sup \{ \alpha > 0 \; ; \; L_0 \otimes \mathcal O(- \alpha D) \; \text{is effective} \;  \}$ 
\end{defi}
From what has just been said, we have $\alpha_{\mathrm{eff}} > 0$.
\medskip

Besides, according to \cite[Proposition 4.2]{HT06}, there exists $\alpha \in \mathbb Q^\ast_+$ such that the stable base locus of the $\mathbb Q$-divisor $L_0 \otimes \mathcal O(- \alpha D)$ is included in $D$. We then let

\begin{defi} $\alpha_{\mathrm{base}} = \sup \left\{ \alpha > 0 \; ; \; \mathbb B \left( \, L_0 \otimes \mathcal O(- \alpha D)\, \right) \subset D \right\}$.
\end{defi}

Actually, it is shown in \cite[Proposition 4.2]{HT06} that $\alpha_{\mathrm{base}} \geq \frac{1}{\binom{n+1}{2} + 2}$ in general. We will see later on that for some particular toroidal compactifications, we can have better lower bounds on both $\alpha_{\mathrm{base}}$ and $\alpha_{\mathrm{eff}}$ (see Propositions \ref{propbaktsi}, \ref{weissauer86} and \ref{grushevsky}).

\subsubsection{Algebraic hyperbolicity. Proof of Theorem \ref{thmlincomb}}

Let us now move on to the proof of the first point of Theorem \ref{thmpartial}. It is actually an immediate consequence of Theorem \ref{thmlincomb} ; here is the proof of this last result.

\begin{proof}[Proof of Theorem \ref{thmlincomb}] Suppose that $L_\alpha$ is effective, and let $V \subset \overline{X}$ be a subvariety of dimension $q \geq p$ such that $V \not\subset \mathbb B(L_\alpha) \cup D$. Choose a rational number $\lambda \in ]0, C_p \alpha[$, and consider a resolution of singularities $\widetilde{V} \overset{j}{\longrightarrow} V \subset  \overline{X}$ such that $F = j^\ast D$ has simple normal crossing support.

By assumption, there exist an integer $m \in \mathbb N$ and a section $s \in H^0 \left( \overline{X}, L_0^{\otimes m} \otimes \mathcal O( - m \alpha D) \right)$, such that $s|_{V}$ is non-zero. Moreover, since $\alpha > \frac{\lambda}{C_p}$, we can find a real number $\beta \in \left] \frac{\lambda}{\alpha m}, \frac{C_p}{m} \right[$.  We define a singular metric $\widetilde{h}$ on $\mathcal O\left( K_{\widetilde{V}} + (1 - \lambda) F_{red} \right)^\ast$, in the following way:
\begin{equation} \label{defimetric}
\widetilde{h} = \norm{s}^{2 \beta} \mathrm \det (j^\ast h_{\Berg}),
\end{equation} 
where the norm $\norm{\cdot}$ is induced by $h_{\Berg}$ on $L_0^{\otimes m}$. In particular, $s$ is seen as a section vanishing at order $m \alpha$ over $D$. 

To conclude the proof, we will show that the dual metric $\widetilde{h}^\ast$ satisfies the hypotheses that permit to apply the main result of \cite{bou02}.

\begin{lem} \label{lembound} The metric $\widetilde{h}$ is locally bounded everywhere on $\widetilde{V}$.
\end{lem}
\begin{proof}[Proof of Lemma \ref{lembound}]
We first prove that $j^{-1}(h_{\Berg})$ has Poincaré growth near the simple normal crossing divisor $F_{red}$, in the sense of \cite{mum77}. The following argument comes directly from \cite{brucad17}, so we only mention it briefly for completeness: it is a simple application of the Ahlfors-Schwarz lemma, combined with the fact that $h_{\mathrm{Berg}}$ has negative holomorphic curvature.
\medskip

The metric $h_{\Berg}$ has negative holomorphic sectional curvature, bounded from above by $- \gamma$. Since the holomorphic sectional curvature decreases on submanifolds, this is also true for $j^\ast h_{\Berg}$, at the points where the latter is regular. Moreover, $j^\ast h_{\Berg}$ is locally bounded on $\widetilde{V} \setminus j^{-1}(D)$. Thus, the usual extension lemma for psh functions across analytic subsets shows that, on $\widetilde{V} \setminus j^{-1}(D)$, the singular metric  $j^\ast h_{\Berg}$ has negative holomorphic sectional curvature bounded from above by $- \gamma$ ; this means that the metric induced on any disk has negative curvature uniformly bounded from above, \emph{in the sense of currents}. We can now use the Ahlfors-Schwarz lemma (see for example \cite{dem12a}) for singular metrics to conclude that $j^\ast h_{\Berg}$ has Poincaré growth near the simple normal crossing divisor $F_{red}$. 
\medskip

The previous argument implies that $\det(j^\ast h_{\Berg})$, \emph{seen as a metric on $\mathcal O_{\widetilde{V}}(K_{\widetilde{V}} + F_{red})^\ast$} is locally bounded. Thus, if $w$ is a local equation for $F_{red}$, and if $\sigma$ is a local frame for $\mathcal O_{\widetilde{V}} \left( K_{\widetilde{V}} + (1 - \lambda) F_{red} \right)^\ast$, we have:
$$
\norm{\sigma}^2_{\mathrm{det} (j^\ast h_{\Berg})} \leq C \frac{1}{\abs{w}^{2 \lambda}},
$$
for some constant $C$. Recall that $s$ vanishes at order $m \alpha$ near $D$, and that the norm $\norm{\cdot}$ on $L_0^{\otimes m}$ induced by $h_{\Berg}$ has logarithmic growth near $D$, by Lemma \ref{loggrowth}. Consequently, the smooth function $\norm{s}$ is bounded from above as
$$
\norm{s} \leq C_1 \,\abs{w}^{m \alpha} \, \abs{\log \abs{w}}^k
$$
for some $k > 0$, and we see by the definition \eqref{defimetric} that since $\beta m \alpha > \lambda$, the metric $\widetilde{h}$ is locally bounded everywhere. This ends the proof of the lemma.
\end{proof}

\begin{lem} The singular metric $\widetilde{h}^\ast$ has positive curvature $i\Theta( \widetilde h^\ast)$, in the sense of currents, and its absolutely continuous part with respect to the Lebesgue measure satisfies 
\begin{equation} \label{ineqcurvature}
i \Theta(\widetilde{h}^\ast)^{ac} \geq (C_p - \beta m) \, \omega_{\mathrm{Berg}}.
\end{equation}
\end{lem}
\begin{proof}
The curvature of $\widetilde{h}$ at any point of $V_{\mathrm{reg}} \setminus D$ can be bounded as:
\begin{align} \nonumber
i \Theta(\widetilde{h}) (v, \overline{v}) & = i \beta j^\ast \Theta \left( \left( \det h^\ast_{\Berg} \right)^m \right) (v, \overline{v}) + i \Theta (\det (j^\ast h_{\Berg})) (v, \overline{v}) \\ \nonumber
	& = \left[ - \beta m \, j^\ast  \mathrm{Ric}(h_{\Berg}) + \mathrm{Ric} (j^\ast h_{\Berg}) \right] (v, \overline{v})  \\ \nonumber
	& \leq \left( \beta m - C_{q}\right) \norm{v}^2  \\ \label{ineqcurvature1}
	& \leq \left( \beta m - C_p \right) \norm{v}^2.
\end{align}
where at the third line we used Proposition \ref{propricci} and our normalization hypothesis on $h_{\Berg}$, while at the fourth line we used Proposition \ref{ineqprop}. Thus, since $\beta m < C_p$, the metric $\widetilde{h}$ has negative curvature at the points where it is smooth. This means that if locally $\widetilde{h}^\ast \overset{loc}{=} e^{- \varphi}$, the weight $\varphi$ is psh on its smooth locus.
\medskip

Thus, we see that the singular metric $\widetilde{h}$ on $ \mathcal O_{\widetilde{V}} (K_{\widetilde{V}} + (1 - \lambda) F_{red})$, is locally bounded everywhere, and has negative curvature outside the divisor $j^{-1}(D \cup V_{\mathrm{sing}})$. Thus, the local weights $\varphi$ for $\widetilde{h}^\ast$ are psh on their smooth locus, and locally bounded. By the usual extension theorem for plurisubharmonic functions across an analytic subset, they are psh on their domain of definition. This means that $\widetilde{h}^\ast$ has positive curvature in the sense of currents. Moreover, dualizing \eqref{ineqcurvature1} yields the wanted inequality, since $V_{reg} \setminus D$ is of full Lebesgue measure in $\widetilde{V}$.
\end{proof}

Now, we can use \cite{bou02}, to bound $\mathrm{vol} (K_{\widetilde{V}} + (1 - \lambda) F_{red})$ from below as
$$
\mathrm{vol}(K_{\widetilde{V}} + (1 - \lambda) F_{red}) \geq \int_{\widetilde{V}} \left( \frac{i}{2\pi} \Theta(\widetilde{h}^\ast )^{ac} \right)^{p} 
$$
By the inequality \eqref{ineqcurvature}, this implies that
$$
\mathrm{vol} \left(K_{\widetilde{V}} + (1 - \lambda) F_{red} \right)  \geq \left( \frac{C_p - \beta m}{2 \pi} \right)^q \int_{V_{\mathrm{reg}} \setminus \left( D \cup \left\{s=0 \right\} \right)} \omega_{\Berg}^q,
$$
and, making $\beta \longrightarrow \frac{\lambda}{ \alpha m}$, we obtain
$$
 \mathrm{vol}(K_{\widetilde{V}} + (1 - \lambda) F_{red}) \geq \left( \frac{C_p - \frac{\lambda}{\alpha}}{2 \pi} \right)^q \int_{V_{\mathrm{reg}} \setminus D} \omega_{\Berg}^q,
$$ 
which proves Theorem \ref{thmlincomb}.
\end{proof}
\medskip

\begin{rem} \label{remguenancia} In the case $p = 1$, we can follow a different approach using the recent work of Guenancia \cite{Guenancia18}. If we suppose that $L_\alpha$ is effective for some $\alpha > \frac{1}{C_1}$, then, with the same notations as before, we can choose a section $s \in H^0 \left( \overline{X},  L_0^{\otimes m} \otimes \mathcal O (- m \alpha D) \right)$, and we can construct the following singular metric on $T_{\widetilde{V}}$, in a fashion similar to \eqref{defimetric}:
$$
\widetilde{h} = \norm{s}^{2 \beta} j^\ast h_{\mathrm{Berg}}, 
$$
with $\norm{\cdot}$ having the same meaning, and $\beta$ being taken as above. 

Then, by an argument similar to Lemma \ref{lembound}, we see that $\widetilde{h}$ is locally bounded everywhere. Moreover, a simple computation shows that on its regular locus, $\widetilde{h}$ has negative holomorphic sectional curvature, bounded from above by a negative constant. By \cite[Proof of Theorem 4.4, Step 6]{Guenancia18}, this implies that $V$ is of general type.
\end{rem}

\subsubsection{$p$-measure hyperbolicity. End of the proof of Theorem \ref{thmpartial}}

Before proving the second claim of Theorem \ref{thmpartial}, we need to recall some classical definitions concerning measure hyperbolicity, that we can find with more details in \cite{dem12a}. On any complex manifold $X$, we can define several intrinsic infinitesimal metrics, whose properties permit to characterize the hyperbolicity properties of $X$. 

\begin{defi} Let $X$ be a complex manifold, and let $\xi = v_1 \wedge ... \wedge v_p$ be a decomposed $p$-vector based at some point $x \in X$. The \emph{Kobayashi-Eisenman pseudometric} of $v$ is the real number
$$
\mathbf{ e}^p(\xi) = \min \left\{ \lambda > 0 ; \exists f : \mathbb B_p \longrightarrow X, \; f(0) = x, \lambda f_\ast (\tau_p) = \xi \right\},
$$
where $\tau_p = e_1 \wedge ... \wedge e_p \in \bigwedge^p T_{\mathbb B^p, 0}$ is the canonical $p$-vector based at $0 \in \mathbb B^p$.
\end{defi}

When $p = 1$, this coincides with the Kobayashi metric. Following the construction of the Kobayashi pseudodistance (see \cite{kob76}), we can form various notions of $p$-measure hyperbolicity modulo a subset. We will deal with the following strong version. 

\begin{defi} \label{defiinfmeashyp} We say a complex manifold $X$ is \emph{infinitesimally $p$-measure hyperbolic} modulo a subset $Z$ if there exists a \emph{positive} continuous metric $\mathbf{\upsilon}$ on the space of all decomposed $p$-vectors on $X \setminus Z$ such that
$$
\mathbf{e}^p(\xi) \geq \mathbf{\upsilon}(\xi)
$$
for any such $p$-vector $\xi$.
\end{defi}

We can now complete the proof of the transcendental part of our main result.

\begin{proof} [End of the proof of Theorem \ref{thmpartial}]
The basic idea of the proof is to use the Ahlfors-Schwarz lemma to bound the size of holomorphic balls $\mathbb B^p \longrightarrow \overline{X}$. To do this, we construct a metric on these balls, of the same form as \eqref{defimetric}. Let us describe how to define such a metric.  
\medskip

Let $x \in \overline{X} \setminus \left( D \cup \mathbb B(L_\alpha) \right)$, and let $\xi$ be a non-zero $p$-vector at $x$. We consider a holomorphic map $f : \mathbb B^p \longrightarrow \overline{X}$ and a real number $\lambda > 0$, such that $ \lambda f_\ast( \tau_0) = \xi$. Then, since $x \not\in \mathbb B(L_\alpha)$, we can, as in the proof of Theorem \ref{thmlincomb}, find an integer $m \in \mathbb N$ and a section $s$ of $L_0^{\otimes m}$, vanishing at order $m \alpha$ along $D$, such that $s(x) \neq 0$. Now, pick a real number $\beta$ as before, and define the singular metric
\begin{equation} \label{modmetriccp}
h_0 = \norm{s}^{\frac{2 \beta}{p}} j^\ast h_{\Berg}
\end{equation}
on $\mathbb B^p$. 

Now, we show that $\det h_0$ has negative curvature in the sense of currents. By the same argument as in Lemma \ref{lembound}, we see that $\widetilde{h} = \mathrm{det} (h_0)$ is locally bounded everywhere on $\mathbb B^p \setminus \Sigma$, where $\Sigma$ is the set of multiple points of $p^{-1}(D)_{red}$. Note that $\mathrm{codim} \Sigma \geq 2$.  Moreover, $\widetilde{h}^\ast$ has positive curvature at the points where it is smooth, so that $i  \partial \overline{\partial}  \log \det h_0  = i \Theta(\widetilde{h}^\ast) \geq \epsilon j^\ast \omega_{\Berg}$, with $\epsilon = C_p - \beta m$. Now, if we pick a constant $C > \sup_{\overline{X}} \norm{s}^{\frac{2 \beta}{q}}$, the extension theorem for plurisubharmonic functions across analytic subsets implies that 
\begin{equation} \label{ineqahlforsschwartz}
i  \partial \overline{\partial} \log \det h_0  \geq \frac{\epsilon}{C} \omega_{h_0}
\end{equation}
on $\mathbb B^p \setminus \Sigma$, \emph{in the sense of currents}. Since $\mathrm{codim} \Sigma \geq 2$, the function $\log \det h_0$ actually extends as a plurisubharmonic function across $\Sigma$, and the inequality \eqref{ineqahlforsschwartz} holds on the whole $\mathbb B^p$. 

Thus, the Ahlfors-Schwarz lemma (see \cite[Proposition 4.2]{dem12a}) shows that 
$$
\det (\omega_0)(0) \leq \left( \frac{p+1}{\epsilon/C} \right)^p.
$$ 
However, by our definition of $h_0$, we have $\det(\omega_0)(0) = \norm{s(x)}^{2 \beta} \frac{\norm{\xi}^2_{\Berg}}{\lambda^{2}}$. Thus, $\lambda \geq C_0 \norm{\xi}_{\Berg}$, so $\mathbf{e}^p(\xi) \geq C_0 \norm{\xi}_{\Berg}$, with $C_0 = \left( \frac{\epsilon/C}{p+1} \right)^{\frac{p}{2}} \norm{s(x)}^\beta$. It is clear that we can take the same constant $C_0$ if $x$ varies in some small open subset not meeting $\mathbb B(L_\alpha) \cup D$. Comparing with Definition \ref{defiinfmeashyp}, we see that this ends the proof of Theorem \ref{thmpartial}.
\end{proof}
\medskip

We can now state and prove the following more precise version of Theorem \ref{thmbru}.   

\begin{thm} \label{mainthm} Assume that $\Gamma$ is neat, and that $\overline{X}$ is a toroidal compactification of $X$, as constructed in \cite{AMRT}. Let $p \in \llbracket 1, n \rrbracket$ and let $\overline{X'} \longrightarrow \overline{X}$ be a ramified cover, étale on the open part, ramifying at orders higher than some integer $l$ above the boundary. Denote by $D'$ the boundary of $\overline{X'}$. Then, for $l \gg 1$ there exists an algebraic subset $Z \subsetneq \overline{X'}$, such that $\overline{X'}$ is Kobayashi hyperbolic modulo $Z$, and such that any $V \subset \overline{X'}$, with $V \not\subset Z$, is of general type. More precisely,
\begin{enumerate}
\item if $l > \frac{1}{\alpha_{\mathrm{eff}} C_p}$, we can take $Z = Z' \cup D'$, for some $Z' \subsetneq \overline{X'}$.

\item if $l > \frac{1}{\alpha_{\mathrm{base}} C_p}$, we can take $Z =D'$.
\end{enumerate}
\end{thm}

\begin{proof} We will only prove the first claim, the second one being similar. Remark that since $\Gamma$ is neat, we have $L_0 = \mathcal O_{\overline{X}} \left( K_{\overline{X}} + D \right)$ by \cite{mum77}. 

 Under the hypothesis $l > \frac{1}{\alpha_{\mathrm{eff} C_p}}$, we can pick a rational number such that $\beta > \frac{1}{C_p}$ and $\frac{\beta}{l} < \alpha_{\mathrm{eff}}$. Then
$$
K_{\overline{X'}} + (1 - \beta) D' \geq \pi^\ast \left( K_{\overline{X}} + \left( 1 - \frac{\beta}{l} \right) D \right),
$$
where the symbol "$\geq$" means that the difference between the two $\mathbb Q$-divisors is effective. Since $\frac{\beta}{l} < \alpha_{\mathrm{eff}}$, the divisor on the right hand side is effective. Thus, the divisor on the left is also effective, and since $\beta > \frac{1}{C_p}$ and $C_p \leq C_q$ for any $q \geq p$, the result follows immediately from Theorem \ref{thmpartial}.
\end{proof}

Now, Theorem \ref{thmbru} is a direct consequence of Theorem \ref{mainthm}, since $\alpha_{\mathrm{base}} \geq \frac{1}{\binom{n+1}{2} + 2}$ by \cite{HT06}. 
\medskip

The Siegel moduli space $\overline{ A_g}$ is a particular toroidal compactification of a quotient of the Siegel half-space $\mathbb H_g$, and in this setting, we also have effective lower bounds on both $\alpha_{\mathrm{eff}}$ and $\alpha_{\mathrm{base}}$, as the following two propositions show. 

\begin{prop} [\cite{weissauer86}] \label{weissauer86}
If $\overline{X} = \overline{A_g}$, we have $\alpha_{\mathrm{base}} \geq \frac{g + 1}{12}$.
\end{prop}

\begin{prop} [\cite{grush09}] \label{grushevsky}
If $\overline{X} = \overline{A_g}$, we have
$$
\alpha_{\mathrm{eff}} \geq \frac{(g+1) \left( 2 g! \, \zeta(2g) \right)^{\frac{1}{g}}}{(2 \pi)^2}. 
$$
\end{prop}

Remark that $A_g$ is in general a singular variety, so we define here $\overline{A_g}$ to be a desingularization of the boundary components of $\overline{A_g}^{BB}$. In the following proof, we state a result about the varieties $A_g(l)$, which are \emph{smooth} when $l \geq 3$, so the results we proved until now about smooth quotients apply. Then, the previous two propositions permit to obtain easily the theorems \ref{thmaghypfull} and \ref{thmaghypmod}. 

\begin{proof}[Proof of Theorems \ref{thmaghypfull} and \ref{thmaghypmod}]
We just have to estimate the two ratios appearing in Theorem \ref{mainthm} in the case $\Omega = \mathbb H_g$. For this domain, we can take $\gamma =\frac{2}{g(g+1)}$. Then, by Proposition \ref{weissauer86}, we have 
$$
\frac{1}{\alpha_{\mathrm{base}} \gamma} \leq \frac{12}{g+1} \frac{g (g+1)}{2} = 6 g.
$$
Similarly, by Proposition \ref{grushevsky}, we find 
\begin{align*}
\frac{1}{\alpha_{\mathrm{eff}} \gamma} & \leq \frac{(2 \pi)^2}{(g + 1)(2 g! \, \zeta(2 g) )^{\frac{1}{g}}} \; \frac{g(g+1)}{2} \\
	& \leq \frac{(2 \pi)^2}{ \left[ 2 \sqrt{2 \pi} g^{g + \frac{1}{2}} e^{-g} \right]^{\frac{1}{g}}} \frac{g}{2} \\
	& \leq \frac{e (2 \pi)^2}{2} < 54,
\end{align*}
where we used the fact that $\zeta(2g) > 1$, as well as Stirling's upper bound $g! \geq \sqrt{2\pi} g^{g+\frac{1}{2}} e^{-g}$. Since $\overline{A_g(n)} \longrightarrow \overline{A_g}$ is a cover ramifying at order at least $n$ on each boundary component if we chose two suitable toroidal compactifications (see \cite{nad89}), both results follow from Theorem \ref{mainthm}.
\end{proof}

\section{Hyperbolicity in intermediate dimensions: cases of $\mathbb B^n$ and $\mathbb H_g$}

We now present some effective results derived from Theorem \ref{thmpartial}, based on an explicit computation of the constants $C_p$ when $\Omega$ is the ball, or the generalized Siegel half-space $\mathbb H_g$. 

\subsection{Case of the ball}

\begin{prop} \label{Cpball} Let $\Omega = \mathbb B^n$. Then, for any $p \in \llbracket 1, n \rrbracket$, we have $C_p = \frac{p+1}{n+1}$.
\end{prop}

\begin{proof}
We endow $\mathbb B^n$ with its usual coordinates $(z_1, ..., z_n)$. Then the Bergman metric, normalized so that $- \mathrm{Ric} (h_{\Berg}) = \omega_{\Berg}$, takes the form
$$
\omega_{\Berg} = i (n+1) \, \,  \frac{(1 - \norm{z}^2) \sum_j dz_j \wedge d\overline{z_j} + \left( \sum_j \overline{z_j} dz_j\right) \wedge \left( \sum_j z_j d \overline{z_j} \right)}{(1 - \norm{z})^2}.
$$
Let $V \subset T_{\mathbb B^n, 0}$ be a subspace of dimension $p$. By letting $U(n)$ acting on $\mathbb B^n$, we can assume $V = \mathrm{Vect} \left( \frac{\partial}{\partial z_1}, ..., \frac{\partial}{\partial z_p} \right)$. Now, a direct computation yields
$$
i \Theta(h_{\mathrm{Berg}})_0 = i \overline{\partial}\partial h_{\mathrm{Berg}} (0) = - i \left[ \sum_{j=1}^n dz_j \wedge \overline{dz_j} \cdot \mathbb I_n + {}^t X \wedge \overline{X} \right],
$$
where $X$ is the row vector $(dz_1, ..., dz_n)$. Restricting to $V$, we find
\begin{align*}
i \,j^\ast \mathrm{Tr}^{V} \left. \left[ \Theta(h_{\mathrm{Berg}}) \right] \right|_{V}^{V} & = - p \, j^\ast \left( i \sum_{j=1}^n dz_j \wedge \overline{dz_j} \right) + i \mathrm{Tr}\left( {}^t X \wedge \overline{X} \right)|_{V}^{V} \\
	& = - i (p +1 ) \sum_{j=1}^p dz_j \wedge \overline{dz_j} \\
	& = - \frac{p+1}{n+1} \, \left. \omega_{\mathrm{Berg}, 0} \right|_V.
\end{align*}
This gives the result, by definition of $C_p$. 
\end{proof}

We can now use the following result of Bakker and Tsimerman to deduce Theorem \ref{thmball} and Corollary \ref{corolball}.

\begin{prop} [Bakker-Tsimerman \cite{baktsi15}] \label{propbaktsi} Assume $\Omega = \mathbb B^n$ and $\Gamma$ has only unipotent parabolic isometries. Then $\alpha_{\mathrm{eff}} \geq \alpha_{\mathrm{base}} \geq \frac{n+1}{2 \pi}$.
\end{prop}

Actually, when $\Gamma$ has only unipotent parabolic isometries, Mok \cite{mok12} gives an explicit description of a \emph{minimal} toroidal compactification of $X$. Then, Bakker and Tsimerman show in \cite{baktsi15} that for such a minimal compactification, and for any $0 < \lambda < \frac{n+1}{2\pi}$, the divisor $K_{\overline{X}} + (1 - \lambda) D$ is ample.
\medskip

Now, to prove Theorem \ref{thmball}, it suffices to combine Theorem \ref{mainthm} with the previous Proposition \ref{Cpball} and Proposition \ref{propbaktsi}. Corollary \ref{corolball} follows easily.

\subsection{Case of the Siegel half-space $\mathbb H_g$} \label{siegelcase}

Let us now deal with the case where $\Omega = \mathbb H_g$, for some $g \in \mathbb N \setminus\{ 0, 1 \}$. The proof of Corollary \ref{corolmgn} will be given at the end of the section.
\medskip

The hermitian symmetric space $\mathbb H_g$ is the classical hermitian symmetric space $D_g^\mathrm{III}$ as denoted in \cite{mok89}, and it can be embedded in $D_{g,g}^{\mathrm{I}} \cong \left\{ M \in \mathcal M_g(\mathbb C) \; ; \; \norm{M}^2 < 1 \right\}$ as
$$
\mathbb H_g \cong D_g^{\mathrm{III}} \cong \left\{ M \in \mathcal{M}_g(\mathbb C) \; ; \; {}^t M = M \; \text{and} \; \norm{M}^2 < 1 \right\}.
$$

We will first determine the values of the constants $C_p$ when $g \geq 8$. To simplify the computations, we will change our normalization on the Bergman metric: we let $h$ be the standard Kähler-Einstein metric on $\mathbb H_g$, with the normalization $\mathrm{Ric}(h) = - (g+1) \omega_h$.

Before we start the proof, and introduce the function which will allow us to compute the constants $C_p$ (see Definition \ref{defiF}), let us recall a few facts concerning the bisectional curvature on $\mathbb H_g$. 
\medskip

The tangent space to $0 \in D_g^{\mathrm{III}}$ is isomorphic to the vector space of symmetric square matrices $S_g(\mathbb C)$, and  the metric $h$ is determined by the norm 
$$
\norm{X}_h^2 = \mathrm{Tr} (X \overline{X}). 
$$

For any $X, Y \in S_g(\mathbb C)$, we let $B_0 (X, Y) = i \Theta(X,  \overline{X}, Y, \overline{Y})$, where $\Theta$ is the curvature form of the metric $h$. Then, by \eqref{defibisectional}, we have:  $B_0 (X, Y) = B(X, Y) \norm{X}^2 \norm{Y}^2$.

According to \cite{mok89}, the function $B_0$ can be computed as follows: 
\begin{equation} \label{bisectexpr}
\forall X, Y \in S_g(\mathbb C), \; B_0(X,Y)  = - 2 \norm{ X \overline{Y}}^2,
\end{equation}
with the natural identification $T_0 D_g^{III} \cong S_g(\mathbb C)$. Let $X \in T_0 D_g^{III} \cong S_g(\mathbb C)$. Still by \cite{mok89}, by letting $U(g)$ act on $T_0 D_g^{III}$, we can assume that
$$
X = \mathrm{diag}(\alpha_1, ..., \alpha_g),
$$
where $\alpha_1, ..., \alpha_g$ are real numbers satisfying $\alpha_1 \geq ... \geq \alpha_g \geq 0$. For each $i$, we let $m_i = \alpha_i^2$.

The eigenspaces of the quadratic form $Y \in \mathcal S_g(\mathbb C) \longrightarrow -B_0(X, Y)$ with respect to the hermitian metric $h$ admit the following simple description. For $i,j \in \llbracket 1, g \rrbracket$, let $E_{i,j}$ denote the standard elementary matrix at the $i$-th row and $j$-th column. Then, we see directly from the expression \eqref{bisectexpr} that the quadratic form $-B_0(X, \cdot)$ admits the following eigenvalues.
$$
\left\{ \begin{array}{cc}
	2 m_i & \text{with multiplicity 1 for each} \; i \in \llbracket 1, g \rrbracket, \; \text{with unitary eigenvector}\; E_{i,i}\\
 	m_i + m_j & \text{with multiplicity 1 for each}  \; i,j \in \llbracket 1, g \rrbracket \; \text{with} \; i < j, \vspace{-0.75em} \\
	& \; \text{and with unitary eigenvector}\; F_{i,j} = \frac{1}{\sqrt{2}} \left( E_{i,j} + E_{i,j} \right).
	\end{array} \right.
$$

For each $p \in \llbracket 1, \frac{g(g+1)}{2} \rrbracket$ and each $X \in \mathcal S_n(\mathbb C)$ of norm $1$, we introduce the constant
$$
D_{X,p} = \min_{V \ni X, \dim V = p} \mathrm{Tr}_{V}^{V} \left( -B_0(X, \cdot)|_V \right),
$$
where the trace is taken with respect to the hermitian metric $h$ on the subspaces $V$.

Remark that because of \eqref{defCxvp}, we have 
$$
C_{0, X, p} = \frac{D_{X,p}}{D_{X, n}},
$$
where $n = \frac{g(g+1)}{2}$ is the dimension of $\mathbb H_g$. With our choice of normalization for $h$, we have $D_{X, n} = g + 1$ for any $X$ with $\norm{X} = 1$. 

Thus, to find the value of $C_p$, it suffices to determine the value of $D_p = \min_{\norm{X} = 1} D_{X, p}$; the constant $C_p$ will then be given by $C_p = \frac{D_p}{g+1}$. For a unitary tangent vector $X$, we can also write $D_{X, p}$ as follows
$$
D_{X,p} = -B_0(X,X) + \min_{V \subset X^{\perp}, \dim V = p - 1} \mathrm{Tr}_{V}^{V} \left( -B_0(X, \cdot)|_V \right).
$$

This second term on the right hand side is actually equal to the sum of the $p-1$ smallest eigenvalues of the quadratic form $-B_0(X, \cdot)$, restricted to $X^\perp$.
\medskip

We will need to consider a certain set of combinatorial data, on which we will define the function $\mathcal F$, whose minimum will be be the constant $D_p$, as stated in Proposition \ref{propeqDp}. For this, we let $T = \left\{ (i,j) \in \llbracket 1, g \rrbracket^2 ; i \leq j \right\}$, and let $\mathcal S_p$ be the set of subsets of $T$ of cardinal $p-1$, with less than $g-1$ elements on the diagonal. We also introduce the simplex $\Delta_g = \left\{ (m_1, ..., m_g) \in \mathbb R_+^g \; ; \; \sum_i m_i = 1 \; \text{and} \; m_1 \geq ... \geq m_g \right\}$.

We can now define the function $\mathcal F$.

\begin{defi} \label{defiF} 
For any $\Gamma \in \mathcal S_p$ and any $\underline{m} \in \Delta_g$, we define the following number:
\begin{displaymath}
\mathcal F(\underline{m}, \Gamma) = \begin{cases}\itemsep=3em
	\; 2 + \sum_{\underset{j > i}{(i,j) \in \Gamma}} \; (m_i + m_j)   \; \; \; \text{if} \; \Gamma \; \text{has} \; g-1 \; \text{diagonal elements} \\
	\; 2 \sum_{i = 1}^g m_i^2 + \sum_{(i,j) \in \Gamma} \; (m_i + m_j)  \; \; \; \text{otherwise} \\ 
	\end{cases}
\end{displaymath}
\end{defi}

\begin{prop} \label{propeqDp} We have
$$
D_p \; = \; \min \left\{ \mathcal{F} (\underline{m}, \Gamma) \; ; \; \underline{m} \in \Delta_g , \; \Gamma \in \mathcal S_p \right\}.
$$
\end{prop}

We break down the proof of this proposition into a series of lemmas. We first show a lower bound for $D_p$. The converse upper bound will proved in Lemma \ref{lemfinal}.

\begin{lem} $D_p \geq \min \left\{ \mathcal{F} (\underline{m}, \Gamma) \; ; \; \underline{m} \in \Delta_g , \; \Gamma \in \mathcal S_p \right\}$.
\end{lem}
\begin{proof}
Let $X = \mathrm{diag}(\alpha_1, ..., \alpha_g)$ be a real unitary diagonal matrix with $\alpha_1 \geq ... \geq \alpha_g \geq 0$, realizing the minimum $D_p = D_{X, p}$. For each $i$, let $m_i = \alpha_i^2$ and let $\underline{m} = (m_1, ..., m_g)$.
\medskip

Let $S(X)_{p-1}$ be the set of the $p-1$ smallest eigenvalues of $-B(X, \cdot)|_{X^\perp}$. Since the space $E^{\mathrm{diag}}$ of diagonal matrices and the space $E^{\mathrm{off}} = \mathrm{Span}(F_{i,j})_{i < j}$ are orthogonal with respect to the quadratic form $ Y \mapsto B_0(X, Y)$, the eigenspaces corresponding to the elements of $S(X)_{p-1}$ split accordingly to the direct sum $X^\perp = (E^{\mathrm{diag}} \cap X^\perp) \oplus E^{\mathrm{off}}$. Let $k$ be the number of eigenvalues of $S(X)_{p-1}$ whose eigenspace is included in $E^{\mathrm{diag}}$, counted with multiplicities. We have to cases to deal with, depending on the value of $k$.
\medskip

Assume first that $k \leq g-2$. Then, by definition of $D_{X, p}$, we have
\begin{align*}
D_{X, p} & = - B_0(X, X) + \sum_{\lambda \in S(X)_{p-1}} \lambda \\
	 & =  2 \sum_{i = 1}^{g} \alpha_i^4 + \sum_{\lambda \in S(X)_{p-1}} \lambda.
\end{align*}
 
 Then, the sum of the elements of $S(X)_{p-1}$ is bounded from below by the sum of the $p-1$ smallest eigenvalues of $B_0(X, \cdot)$, with multiplicities. The eigenvalues of $B_0(X, \cdot)$ being equal to the $m_i + m_j$, for $i \leq j$, we can choose $\Gamma \subset T$ such that these $p-1$ eigenvalues are
$$
\left\{ m_i + m_j \; ; \; (i,j) \in \Gamma \right\},
$$
with multiplicities. Since the $m_i$ are non-increasing, and since $\mathrm{Card}(\Gamma) \leq  \frac{g(g+1)}{2} - 1$, we see easily that we can suppose that $\Gamma$ has less than $n-1$ diagonal elements (or else we could shift one of these elements the right to get a new set $\Gamma'$ with $\mathcal F(\underline{m}, \Gamma') \leq \mathcal F(\underline{m}, \Gamma)$ ). Consequently, $\Gamma \in \mathcal S_p$.

But then we have

$$
D_{X,p} \geq 2 \sum_i m_i^2 + \sum_{(i,j) \in \Gamma} (m_i + m_j) = \mathcal F(\underline{m}, \Gamma).
$$
 Since $D_p = D_{X, p}$, this gives the result in that case.
\medskip

Suppose now that $k = g-1$. Then, if $\Sigma \subset S(X)_{p-1}$ denotes the set of all diagonal eigenvalues of $-B_0(X, \cdot)|_{X^\perp}$, we have $\mathrm{Card} (\Sigma) = g-1 = \mathrm{dim} (X^\perp \cap E^{\mathrm{diag}})$, so
\begin{align*}
-B_0(X, X) + \sum_{\lambda \in \Sigma} \lambda & = -B_0(X,X) - \mathrm{Tr}_{E^{\mathrm{diag}} \cap X^\perp} B_0(X, \cdot)|_{X^\perp} \\
	& = - \mathrm{Tr}_{E^{\mathrm{diag}}} B_0(X, \cdot)|_{E^{\mathrm{diag}}} \\
	& = 2 \sum_{i=1}^g m_i \\
	& = 2 \norm{X}^2 = 2.
\end{align*}
All the elements of $S(X)_{p-1}$ which are not diagonal have their eigenspaces included in $E^{\mathrm{off}}$: they are of the form $m_i + m_j$, for some $j > i$. Let $\Gamma_+$ be the set of all the $(i,j)$ appearing this way, and let $\Gamma = \Gamma_+ \cup \left\{ (i,i) ; i \neq 1 \right\}$. Clearly, $\Gamma \in \mathcal S_p$. Thus we have
$$
D_{X,p} = 2 + \sum_{(i,j) \in \Gamma_+} (m_i + m_j) = \mathcal F(\underline{m}, \Gamma),
$$
which gives the result.
\end{proof}

Now, we choose a minimizing pair for the function $\mathcal F$, whose properties will be studied in the rest of the section.

\begin{defi} \label{defim0g0} Let $(\underline{m_0}, \Gamma_0) \in \Delta_g \times \mathcal S_p$ minimizing $(\underline{m}, \Gamma) \mapsto \mathcal F( \underline{m}, \Gamma)$, and we suppose in addition that $(\underline{m_0}, \Gamma_0)$ satisfies the following conditions, stated in order of decreasing priority.
\begin{enumerate} \itemsep=0em
\item $\underline{m_0}$ has the largest possible $l$ such that $(m_0)_{g-l+1} = ... = (m_0)_g = 0$ ;
\item among all sets $\Gamma$ minimizing the function $\Gamma \mapsto \mathcal F \left(\underline{m_0}, \Gamma \right)$, $\Gamma_0$ has the largest possible number of diagonal elements, the off-diagonal elements are maximal for the lexicographic order, and the diagonal elements have the largest possible indices.
\end{enumerate}
\end{defi}

In the rest of the section, $l$ will have the meaning given in the first point of Definition \ref{defim0g0}.

\begin{rem}
Since the elements of $\underline{m_0}$ are non-increasing, then $\Gamma_0$ is \emph{saturated}, in the following sense: for any off-diagonal element $(i,j) \in \Gamma_0$, we must have $\left\{ (s,t) ; s \geq i, t \geq j \right\} \subset \Gamma_0$. Indeed, if it were not the case, we could replace $(i,j)$ by an element $(s,t)$ such that $m_i + m_j \geq m_s + m_t$ and $(k,l) \geq (s,t)$, with respect to the lexicographic order. Also, if $(i,i) \in \Gamma_0$, we must have $\left\{ (j,j) ; j \geq i \right\} \subset \Gamma_0$.
\end{rem}

In the next lemma, we show a first restriction on the shape of $\Gamma_0$. In the rest of the section, $k$ will denote the number of diagonal elements of $\Gamma_0$. Then the diagonal elements of $\Gamma_0$ are the $(i,i)$, for $i \geq g-k+1$.

\begin{lem} \label{lemlequalsk} We have $l = k$. 
\end{lem}
\begin{proof} We have two cases to study, depending on the value of $k$.
\medskip
 
\begin{bfseries}Case 1:\end{bfseries} Suppose that $k \leq g-2$.

Let us show that $k \leq l$. Suppose by contradiction that $l > k$, which means that $m_{g-k} = 0$. If $\Gamma_0$ has an off-diagonal element, we can move it to $(g-k, g-k)$, to obtain a new set $\Gamma$, with one more diagonal element than $\Gamma_0$, and such that
$$
\mathcal F(\underline{m_0}, \Gamma_0) \geq \mathcal F(\underline{m}, \Gamma).
$$
This is absurd by our choice of $\Gamma_0$, since $\Gamma_0$ is supposed to have the largest possible number of diagonal elements. Thus, $\Gamma_0$ has no off-diagonal element, and $\underline{m_0}$ minimizes
$$
\underline{m} \longmapsto \mathcal F(\underline{m}, \Gamma_0) = 2 \sum_{i=1}^g m_i^2 + 2 \sum_{i = g-k+1}^{g} m^2_i.
$$
Since $g-k \geq 2$, this implies easily (for example by Lemma \ref{lemvanishing1} below) that $(m_0)_i = \frac{1}{g-k}$ for $1 \leq i \leq g-k$, and is equal to $0$ otherwise, so $(m_0)_{g-k} \neq 0$. This is again absurd, so $l \leq k$.

Now, we show that $l \geq k$. Then, $\underline{m_0}$ minimizes $\mathcal F( \cdot , \Gamma_0)$, which takes the form
\begin{align}
\mathcal F(\underline{m_0}, \Gamma_0) & = 2 \sum_{i=1}^g (m_0)_i^2 + \sum_{(i,j) \in \Gamma_0} \left( (m_0)_i + (m_0)_j \right) \label{exprF} \\ \nonumber
	& = 2 \sum_{i = 1}^g (m_0)_i^2 + 2 \sum_{i = g-k+1}^g (m_0)_i + \sum_{i = 1}^{g} c_i (m_0)_i
\end{align}
where $c_i$ is the number of times the integer $i$ appears in a off-diagonal element of $\Gamma_0$. These numbers are non-decreasing $\Gamma_0$ since is saturated. Thus, if $b_i$ is the number of times $(m_0)_i$ appears in the sum in $\eqref{exprF}$, we have $b_i = 2 \cdot \mathbbm{1}_{i \geq g - k + 1} + c_i$, and
$$
\sum_{i = 1}^{g-k+1} (b_{g-k+1} - b_{i}) \; \geq \sum_{i=1}^{g-k} 2 \; \geq \; 4
$$
since $g-k \geq 2$. Thus, by Lemma \ref{lemvanishing1} below, we have $(m_0)_{g-k+1} = ... = (m_0)_{g} = 0$. This shows that $l \geq k$.
\medskip

\begin{bfseries}Case 2:\end{bfseries} Suppose that $k = g-1$.

Then the $g$-uple $\underline{m_0}$ minimizes the function
$$
\underline m \in \Delta_g \longmapsto \mathcal F(\underline{m}, \Gamma_0) = 2 + \sum_{\underset{(i,j) \in \Gamma_0}{j > i}} (m_i + m_j).
$$
The latter can be written as
$$
\mathcal F(\underline{m}, \Gamma_0) = 2 + \sum_{i = 1}^{g} a_i m_i,
$$
for some non-decreasing integers $a_i$. Then it is clear that the minimum is reached for $\underline{m_0} = \left(1, 0, ..., 0 \right)$, so $l = g-1$.
\end{proof}

\begin{lem} \label{lemvanishing1} Let $b_1 \leq ... \leq b_g$ be non-negative integers, and let $t$ be the smallest integer such that $\sum_{i = 1}^t \left( b_t - b_i \right) \geq 4$ (we take $t = g$ if there is no such integer). Let $ \underline{m} = (m_i)_{1 \leq i \leq g} \in \Delta_g$ be a minimizer for the quadratic form
$$
Q(x_1, ..., x_g) = 2 \sum_{i=1}^g x_i^2 + \sum_{i = 1}^g b_i x_i.
$$
Then $m_{t+1} = ... = m_g = 0$, and the value of the minimum is 
$$
Q(\underline{m}) = \frac{1}{8 \,t} \left(4 + \sum_{i = 1}^{t} b_i \right)^2 - \frac{1}{8} \sum_{i= 1}^{t} b_i^2. 
$$
\end{lem}
\begin{proof} Assume first that $t = g$, i.e. $\sum_{i=1}^g (b_g - b_i) \geq 4$. We have $Q(X) = 2 \sum_{i=1}^g (x_i + \frac{b_i}{4})^2 - \frac{1}{8} \sum_i^g b_i^2$, so the minimum is realized at the point $R$ which is the projection of $P = \left(-\frac{b_1}{4}, ..., -\frac{b_g}{4} \right)$ on the convex domain $\Delta_g$. The point $R$ is actually the point of $\Delta_g$ which is closest to $\pi_{H}^\perp (P)$, the orthogonal projection of $P$ on the hyperplane $H = \left\{ \sum_i x_i = 1 \right\}$. 

We have $\pi_H^\perp(P) = \left( -\frac{b_1}{4} + \frac{\sum_{i} b_i}{4 g}+ \frac{1}{g}, ..., -\frac{b_g}{4} + \frac{\sum_{i} b_i}{4 g} + \frac{1}{g} \right)$. Since $\sum_{i=1}^g (b_g - b_i) \geq 4$, the last coordinate of $\pi_H^\perp(P)$ is non-positive, so $R$ lies on the hyperplane $\left\{ x_g = 0 \right\}$, thus $m_g = 0$. The proof of the first point follows by decreasing induction on $s$.

To prove the second point, we see that we can suppose that $t = g$,  so we have $\underline{m} = \pi^\perp_H(P)$. The formula follows from a simple computation, using the above expressions for $Q$ and $\pi^\perp_H(P)$. 
\end{proof}

The following lemma ends the proof of Proposition \ref{propeqDp}.

\begin{lem} \label{lemfinal} We have $\mathcal F(\underline{m_0}, \Gamma_0) \geq D_{p}$. 
\end{lem}
\begin{proof}
Assume first that $k \leq g-2$. Let $Y = \mathrm{diag}(\beta_1, ..., \beta_g)$, with $\beta_i = \sqrt{(m_0)_i}$, and let
$$
V = \mathrm{Vect} \left(  \left\{ E_{i,i} \; ; \; (i,i) \in \Gamma_0 \right\} \cup \left\{ F_{i,j} \; ; \; (i,j) \in \Gamma_0 \right\} \right).
$$
Then, since $\beta_{g-k + 1} = ... = \beta_{g} = 0$ by Lemma \ref{lemlequalsk}, we see that $V \subset Y^\perp$. Moreover, by our choice of $V$, $-B_0(Y, \cdot)|_V$ can be diagonalized, with eigenvalues $\beta_i^2 + \beta_j^2 = (m_0)_i + (m_0)_j$, for $(i,j) \in \Gamma_0$. Since $Y$ is unitary, this gives
\begin{align*}
D_{Y, p} & \leq - B_0(Y, Y) + \mathrm{Tr}_V (-B_0(Y, \cdot)|_V) \\
  	 & = 2 \sum_{i= 1}^g (m_0)_i^2 + \sum_{(i,j) \in \Gamma_0} \left((m_0)_i + (m_0)_j \right) \\
	 & = \mathcal{F}(\underline{m_0}, \Gamma_0).
\end{align*}
Thus, $D_p \leq D_{Y,p} \leq \mathcal{F}(\underline{m_0}, \Gamma_0)$.

In the case where $k = g-1$, we proceed exactly as before, using the fact that $\underline{m_0} = \mathrm{diag}(1, 0, ..., 0)$ in that case, as we saw in the proof of Lemma \ref{lemvanishing1}. Then we have $ 2 \sum_i (m_0)_i^2 = 2$, and $\sum_{(i,i) \in \Gamma} (m_0)_i = 0$. Thus, the previous computation gives the same conclusion.
\end{proof}

Now, the only thing left to do is to compute the value $D_p = \mathcal F(\underline{m_0}, \Gamma_0)$. There is one case where the computation is easy, as stated in the next lemma. In the following, we let $k$ have the same meaning as above: we have $\underline{m_0} = \left((m_0)_1, ..., (m_0)_{g-k}, 0, ..., 0 \right)$, with $(m_0)_{g-k} \neq 0$.

\begin{lem} \label{lemvalueDpcase0} If $p \leq \frac{k(k+1)}{2} + 1$, then $D_p = \frac{2}{g-k}$.
\end{lem}
\begin{proof}
In that case, we have $p - 1  \leq \frac{k(k+1)}{2}$, so since $(m_0)_{g - k + 1} = ... = (m_0)_g = 0$, the $p-1$ smallest $(m_0)_i + (m_0)_j$ are all equal to $0$. Thus, $\mathcal F(\underline{m_0}, \Gamma_0) = 2 \sum_{i = 1}^{g-k} (m_0)_i^2$, and $\underline{m_0}$ minimizes this expression in $\Delta_g$. This gives the result.
\end{proof}

In the remaining cases, we are going to find other restrictions on the shape of $\Gamma_0$.

\begin{lem} \label{lemshape} If $p > \frac{k(k+1)}{2} + 1$, then $\Gamma_0 = T_{g-k} \cup \Pi_0$, where $T_{g-k}$ is the upper triangle
$$
T_{g-k} = \left\{ (i,j) ; j \geq i \geq g - k + 1 \right\},
$$
and $\Pi_0 \subset Q_{g-k} :=  \llbracket 1, g-k \rrbracket \times \llbracket g- k + 1, g \rrbracket$. 
\end{lem}
\begin{proof} As we saw before, we have $(m_0)_i + (m_0)_j = 0$ for $i \geq j \geq g - k + 1$, and by hypothesis, $\mathrm{Card}(\Gamma) > \mathrm{Card}(T_{g - k})$. Since $\Gamma_0$ minimizes $\mathcal F(\underline{m_0}, \cdot)$, it follows from the assumptions we made on $\Gamma_0$ that $T_{g-k} \subset \Gamma_0$.

Now, it suffices to show that for any $i \leq j \leq g-k$, then $(i,j) \not\in \Gamma_0$. But for such an $(i,j)$, we have $(m_0)_i + (m_0)_j \geq 2 (m_0)_j \geq 2 (m_0)_{g-k}$. Thus, if $(i,j) \in \Gamma$, we can replace $(i,j)$ by $(g-k, g-k)$ in $\Gamma$, to obtain a new set $\Gamma'$, with one more element on the diagonal, and $\mathcal F(\underline{m_0}, \Gamma') \leq \mathcal F(\underline{m_0}, \Gamma_0)$. This is a contradiction.
\end{proof}

From now on, we suppose that $p > \frac{k(k+1)}{2} + 1$, so $\Gamma_0$ has the shape described in Lemma \ref{lemshape}. For any $i \in \llbracket 1, g-k \rrbracket$, let $a_i$ be the number of elements of $\Pi_0$ on the $i$-th row. Then
\begin{equation} \label{minexpr}
\mathcal F(\underline{m_0}, \Gamma_0) = 2 \sum_{i = 1}^{g-k} (m_0)^2_i + \sum_{i = 1}^{g-k} a_i (m_0)_i.
\end{equation} 
Since the $(m_0)_i$ are in non-increasing order, we see that $\Pi_0$ can only have one row which is partially filled, and that all the subsequent ones must be totally filled (or else we could move some elements in $\Pi_0 \subset \Gamma_0$ to the bottom-right without increasing $\mathcal F(\underline{m_0}, \Gamma_0)$). Moreover, the elements of the row which is not totally filled must be located on the far right of the row, by our assumptions on $\Gamma_0$.
\medskip

The following lemma gives an important restriction on the number of elements of $Q_{g - k} \setminus \Pi_0$.

\begin{lem} \label{lememtpyslots} We must have $\mathrm{Card} \left(Q_{g-k} \setminus \Pi_0 \right) \leq 3$.
\end{lem}
\begin{proof}
Suppose by contradiction that $\mathrm{Card} (Q_{g-k} \setminus \Pi_0) \geq 4$. Then $\sum_{j = 1}^{g-k} (a_{g-k} - a_j) = \mathrm{Card}(Q_{g-k} \setminus \Pi_0) \geq 4$. By Lemma \ref{lemvanishing1}, since $\underline{m_0}$ minimizes \eqref{minexpr}, this implies that $(m_0)_{g-k} = 0$. This is a contradiction.
\end{proof}

\medskip

Now, two types of shapes for $\Pi_0$ are \emph{a priori} possible.

\begin{paragraph}{Case 1. The last row of $\Pi_0$ is filled.} We will see that in this case, since we assumed $g \geq 8$, the only possibility is that $g-k = 1$, so $\mathcal F(\underline{m_0}, \Gamma_0) = g + 1$.
\medskip

In the case 1, we have $a_{g-k} = k$. So, if $g-k = 1$, we see that by Lemma \ref{lemshape} that $\Gamma_0 = T \setminus \left\{ (1,1) \right\}$, and $\underline{m_0} = (1, 0, ..., 0)$. Thus $\mathcal F(\underline{m_0}, \Gamma_0) = g + 1$.
\medskip

The next lemma shows that necessarily $g - k = 1$ in the case 1.

\begin{lem} \label{lemruleout} If $g \geq 8$ and $g-k \geq 2$, the case 1 cannot happen.
\end{lem}

\begin{proof}
We start by ruling out two possibilities, listed below. In each case, the method is the same: we remove a certain number $\delta \leq 4$ of elements to $\Pi_0$, which we choose to be the smallest possible in the lexicographic order of $T$. This gives a new set $\Gamma'$, for which we clearly have $\mathcal F(\underline{m_0}, \Gamma_0) \geq \mathcal F(\underline{m_0}, \Gamma')$. Now, let $\underline{m}$ be a minimizer in $\Delta_g$ for the function $\mathcal F(\cdot, \Gamma')$. We can make our choice of removed elements so that above a certain row indexed by some integer $l \leq g-k$, there is at least $4$ empty slots in $\Gamma'$. Now, by the same argument as in the proof of Lemma \ref{lememtpyslots}, this shows that $m_{l} = ... = m_{g} = 0$. Now, in each one of the following cases, $l$ and $\delta$ can be chosen so that there is enough room in the triangle $\{ g-k \geq j \geq i \geq l \}$ to add back $\delta$ elements to $\Gamma'$, to give a new set $\Gamma_1$. Then
\begin{align*}
\mathcal F(\underline{m_0}, \Gamma_0) & \geq \mathcal F(\underline{m_0}, \Gamma') \\
		& \geq \mathcal F(\underline{m}, \Gamma') \\
		& = \mathcal F(\underline{m}, \Gamma_1).
\end{align*} 
But then $\mathrm{Card}(\Gamma_1) = \mathrm{Card}(\Gamma_0) = p -1$, and $m_{g-k} = 0$, so this gives a contradiction with our previous choice of $\underline{m_0}$. This method shows that the following cases cannot happen:
\begin{enumerate}[(i)] \itemsep=0em
\item $k \geq 4$ and $g- k \geq 4$ ;
\item $k \leq 3$ and $g - k \geq 5$ ;
\end{enumerate}

Now, we treat the case $g-k = 3$. Then $k \geq 5$ because we supposed that $g \geq 8$. By the previous argument, all lines of $\Pi_0$ must be filled. Indeed, if there were a missing element on the first row, we could remove 3 elements on this row, and put them in the triangle $\{ 2 \leq i \leq j \leq 3 \}$, to obtain a new set $\Gamma'$. Then we would achieve a lower bound $\mathcal F(\underline{m}, \Gamma') \leq \mathcal F(\underline{m_0}, \Gamma_0)$, for $\underline{m} = (1, 0, ..., 0)$, which is absurd. Thus, we have $a_1 = ... = a_{g-k} = k$, and $\underline{m_0}$ achieves the minimum of $\mathcal F(\underline{m_0}, \Gamma_0) = 2 \sum_{i = 1}^{g-k} (m_0)_i^2 + k \sum_{i = 1}^{g-k} (m_0)_i$. We see easily that the minimum of this expression is equal to $\mathcal F(\underline{m_0}, \Gamma_0) = \frac{2}{g-k} + k = \frac{2}{3} + k$. However, if we move 3 elements of the first row of $\Gamma_0$ to the triangle $\{ 2 \leq i \leq j \leq 3 \}$, to get a set $\Gamma'$, and if we let $\underline{m} = (1, 0, ..., 0)$, we have
$$
\mathcal F(\underline{m}, \Gamma') = 2 + (k - 3) = k - 1 < \mathcal{F} (\underline{m_0}, \Gamma_0).
$$ 
This is absurd.

The only case left is when $g - k = 2$. Then $a_2 = k$, and $a_1 = k - \delta$, with $\delta \in \{0, 1, 2, 3\}$. In that case $\underline{m_0} = ((m_0)_0, (m_0)_1, 0, ..., 0)$ minimizes 
$$
\underline{m} \in \Delta_g \longmapsto \mathcal F(\underline{m}, \Gamma_0) = 2 \left( m_1^2 + m_2^2 \right) + k - \delta m_1,
$$
among all possible choices of $(m_0)_1 \geq  (m_0)_2 \geq 0$ such that $(m_0)_0 + (m_1)_1 = 1$. We compute easily the value of this minimum: we have $\mathcal F(\underline{m_0}, \Gamma_0) = 2 + k - \left( 1 + \frac{ \delta}{4} \right)^2$. But if we move one element of the last row of $\Gamma$ to the slot $(1, 1)$ to get a new set $\Gamma'$, and if $\underline{m} = (1, 0, ..., 0)$, we have
$$
\mathcal F(\underline{m}, \Gamma') = 2 + (k - \delta - 1) \leq \mathcal F(\underline{m_0}, \Gamma_0),
$$ 
since $\delta \leq 3$. This contradicts our choice of $\underline{m_0}$.
\end{proof}
\end{paragraph}

\begin{paragraph}{Case 2. The last row of $\Pi_0$ is not full, and all the other ones are empty.}
We will determine a finite number of subcases in this situation, all summed up in Proposition \ref{propvalueDpcase2}.

Let $r$ be the number of elements of $\Pi_0$ on this last row. Then we have $r = \abs{\Pi_0} =  p - 1 - \frac{k(k+1)}{2} > 0$. The $g$-uple $\underline{m_0}$ minimizes
$$
\underline{m} \in \Delta_g \longrightarrow \mathcal F(m, \Gamma_0) = 2 \sum_{i = 1}^{g-k} m_i^2 + r \, m_{g-k}.
$$
By Lemma \ref{lemvanishing1}, since $(m_0)_{g-k} \neq 0$, we see that we must have $r (g-k - 1) \leq 3$. This leaves only a finite number of possibilities, for which we can compute the value of $\mathcal F(\underline{m_0}, \Gamma_0)$.

Putting everything together, we obtain the following result.

\begin{prop} \label{propvalueDpcase2} In the case 2, the only possibilities are the one listed in the table of Figure 3, which gives the value of $D_p$ for each case.

\begin{figure}[h]
\centering
\setlength{\tabcolsep}{5pt}
\begin{tabular}{|c|c|c|c|c|}
\hline
 & $g-k = 1$ & $g-k=2$ & $g-k = 3$ & $g-k = 4$ \\
\hline
$r =1$ & \multirow{4}{*}{$r + 2$} & $\frac{23}{16}$ & $\frac{11}{12}$ & $\frac{21}{32}$ \\
\cline{1-1}\cline{3-5}
$r = 2$ & & $\frac{7}{4}$ & \cellcolor{gray} & \cellcolor{gray} \\
\cline{1-1}\cline{3-5}
$r = 3$ & & $\frac{31}{16}$ & \cellcolor{gray} & \cellcolor{gray} \\
\cline{1-1}\cline{3-5}
$r \geq 4$ & & \cellcolor{gray} & \cellcolor{gray} & \cellcolor{gray}  \\
\hline
\end{tabular}
\caption{Possible values of $D_p$ in the case 2}
\end{figure}
\end{prop}
\end{paragraph}

Now, we can find the actual constant $D_p$ when $g \geq 8$. Because of Lemma \ref{lemvalueDpcase0} and Proposition \ref{propvalueDpcase2}, we have $D_p = \frac{2}{g-k}$, where $k$ is the smallest integer such that $p-1 \leq \frac{k(k+1)}{2}$, \emph{except} if $p$ is the cardinal of a set $\Gamma_0$, having one of the shapes discussed in Proposition \ref{propvalueDpcase2}, which then gives the value of $D_p$. 
\medskip

To find the values of the constants $D_p$ when $g < 8$, we only have a finite number of cases to handle. Rather than doing a tedious case by case analysis, we prefer to present a Python procedure to compute these constants. Incidentally, it can also permit to check that the previous computations are correct for a given value of $g$. 

We present these Python functions in an annex to this article. The important functions are \verb?C?, which computes the value of the constant $D_p$ using Proposition \ref{propeqDp}, and \verb?C1?, which takes a pair $(g, p)$, and returns $(g+1)C_p$, where $C_p$ is given by the table of Proposition \ref{propconstants}. We check that \verb?C? and \verb?C1? return the same values for any $2 \leq g \leq 8$, and any $p \leq \frac{g(g+1)}{2}$, which ends the proof of Proposition \ref{propconstants}. We can also check that we obtain the same results for other $g \geq 9$, which tends to confirm our previous computations.
\medskip

At this point, we can prove Corollary \ref{corolmgn}, by using Theorem \ref{mainthm}, combined with Proposition \ref{propconstants}.

\begin{proof} [Proof of Corollary \ref{corolmgn}] For $n \geq 3$, the period map gives a generically immersive map from $M_{g, n}$ to the projective manifold $\overline{A_g(n)}$. Let $p = 3g - 3$ be the dimension of $\mathrm{dim} M_{g,n}$. 

We can use Proposition \ref{propconstants} to determine the value of $C_p$ in this case: it is easy to see that $k = \lfloor \frac{  \sqrt{24 g - 31} - 1 }{2} \rfloor$. Remark that $p - 1 \neq \frac{k(k+1)}{2}$, otherwise we would have $k = \frac{\sqrt{24 g - 31} - 1}{2}$, so $24g - 31$ would be a perfect square. However, this is impossible, since $24 g - 31 \equiv - 1 [\mathrm{mod} \, 3]$. Thus, $r = p - 1 - \frac{k(k+1)}{2} > 0$.

If $g \geq 12$, a simple computation shows that $g - k \geq 5$, so by the table 1, we have $C_p = \frac{2}{g-k-1} \frac{1}{g+1}$. Thus, by Proposition \ref{weissauer86}, we have in that case $\frac{1}{\alpha_{\mathrm{base} C_p}} \leq 6(g -k - 1)$, which gives the first point by Theorem \ref{mainthm}.

If $g \leq 11$, we can compute the constant $C_{3g - 3}$ for each possible value of $g$ (using for example the Python function \verb?C(g, p)? provided in the annex), to give an upper bound for $\frac{1}{\alpha_{\mathrm{base}} C_{3 g - 3}}$ for each $g \in \llbracket 2, 12 \rrbracket$. Then, Theorem \ref{mainthm} permits to conclude in a similar manner.
\end{proof}

To prove Corollary \ref{corolsubv}, we will need to extend the previous results, and especially Theorem \ref{thmpartial}, to the case of \emph{singular quotients} of bounded symmetric domains. This is the purpose of the next section.

\section{Singular quotients of bounded symmetric domains}

We now explain how to extend the previous considerations to a setting where the lattice $\Gamma$ is no longer assumed to be torsion free. We will instead assume that $\Gamma$ has no element fixing some codimension $1$ component. As above, we consider a resolution of the singularities of the Baily-Borel compactification $\widetilde{X} \overset{\pi}{\longrightarrow} \overline{X}^{BB}$.  
\medskip

Let $D = p^{-1}(D^{BB})$ be the boundary divisor of $\widetilde{X}$, where $D^{BB} = \overline{X}^{BB} \setminus X$, and let $E \subset \widetilde{X}$ be the sum of the components of the exceptional divisor lying above $X$. We can assume that $E + D$ has simple normal crossing support. 

Let $d \in \mathbb N$. Following \cite{Tai1982} and \cite{weissauer86} , we will determine a condition on the elliptic elements of $\Gamma$ so that the singular metric induced by $h_{\mathrm{Berg}}$ on $\bigwedge^d T_{\widetilde{X}}$ is locally bounded on $\widetilde{X} \setminus D$.
\medskip

Denote by $F^{(2)}$ the set of crossing points of the support of $E + D$, and consider a point $x \in E \setminus F^{(2)}$. Let $W \cong \mathbbl{\Delta}^n$ be a polydisk centered at $x$ such that $W \cap E = \left\{ w_1 = 0 \right\}$, where $(w_1, ..., w_n)$ are coordinates on $W$. Then, we can find a neighborhood of $W$ which is a resolution of a quotient $\quotient{G}{U}$, where $U \Subset \Omega$ is an open subset and $G$ is a finite group of automorphisms of $U$. Let $U^\circ \subset U$ be the open set where $G$ acts freely, and $W^\circ = W \setminus E$.

Then, we have the following diagram:
$$
  \xymatrix{
      V^\circ = \mathbbl{\Delta}^\ast \times \mathbbl{\Delta}^{n-1} \ar[r]^-{i} \ar[d]^{p} & U^\circ  \ar[d]^{q} \\
      W^\circ = \mathbbl{\Delta}^\ast \times \mathbbl{\Delta}^{n-1} \ar[r]^-{j}     & \quotient{G}{U^\circ} 
	}	
  $$
where $V^\circ$ is any connected component of the fibre product $W^\circ \times_{\tiny{\quotient{G}{U^\circ}}} U^{\circ}$. The group of the étale cover $p$ is cyclic of order $r$ for some $r \geq 1$, and its generator acts by $[1] \cdot (z_1, ..., z_n) = (\gamma z_1, z_2 ..., z_n)$, with $\gamma = e^{\frac{2 i \pi}{r}}$. Moreover, $p$ can be written $p\, (z_1, ..., z_n) = (w_1 = z_1^r, w_2 = z_2, ..., w_r = z_r)$. Remark that both $i$ and $j$ are open immersions, and that $i$ is a morphism of étale covers. Thus, if $g \in G$ is the image of $[1]$ by the induced map $\quotientd{\mathbb Z}{r \mathbb Z} \longrightarrow G$, we must have, for any $(z_1, ..., z_n) \in V^\circ$:
$$
i (\gamma z_1, z_2 ..., z_n) = g \cdot i\, (z_1, ..., z_n).
$$

Since $g^r = \mathrm{Id}_{U}$, by Bochner's linearization theorem, we can suppose that there are coordinates near $\overline{U^\circ}$ such that $g$ has the expression
$$
g \cdot (Z_1, ..., Z_n) = (\gamma^{a_1} Z_1, ..., \gamma^{a_n} Z_n)
$$
for some integers $a_i \in \llbracket 0, r - 1 \rrbracket$. Thus, if we write $i(z) = (F_1(z), ..., F_n(z))$ for $z \in V^0$, we see easily that for any $k$, we have
$$
F_k(z) = z_1^{a_k} G_k (z_1^r, z_2, ..., z_n),
$$ 
for some holomorphic map $G_k : W \longrightarrow U$, where $W = \mathbbl{\Delta} \times \mathbbl{\Delta}^{n-1} \supset W^\circ$.
\medskip

As explained in \cite{weissauer86}, we can now write a condition on $g$ such that for any $d$-vector field $v \in \Gamma \left( W, \bigwedge^d T_W \right)$, the meromorphic vector field $i_\ast p^\ast (v)$ is actually holomorphic.

Let $I \subset \llbracket 1, n \rrbracket$ be a subset of cardinal $d$. Then, we have
\begin{align} \nonumber
p_\ast i^\ast (d Z_I) & = p_\ast i^\ast \left(\bigwedge_{j = 1}^d dZ_{I_j} \right) \\ \nonumber
	& = p_\ast \left( \bigwedge_{j=1}^d (a_k z_1^{a_k - 1} (G_k \circ p) dz_1 + z_1^{a_k} d(G_k \circ p) \right) \\ 
	& = \bigwedge_{j=1}^{d} \left(a_k w_1^{\frac{a_k}{r}} G_k \frac{dw_1}{w_1} + w_1^{\frac{a_k}{d}} d G_k \right). \label{eqaj}
\end{align}
Expending the last expression, we see that if the following condition $(\mathrm{I}_{x, d})$ is satisfied, $p_\ast i^\ast (dZ_I)$ has no pole along ${w_1 = 0}$ for any $I$, or dually, that for any $I$, the $p$-vector field $i_\ast p^\ast \left(\bigwedge_{j \in I} \frac{\partial}{\partial z_j} \right)$ is bounded near $E$ on $U^\circ$.

\begin{crit}{$\mathbf{ I_{x, d}}$}
For the given $x \in E \setminus F^{(2)}$, let $(a_1, ..., a_n) $ and $r$ be as above. Then for any choice of $d$ distinct elements $a_{I_1}, ..., a_{I_d}$, we have
\begin{equation} \label{conditionI}
\sum_{j=1}^{d} a_{I_j} \geq r. \tag*{$(\mathrm I_{x, d})$}
\end{equation}
\end{crit}

Remark that if $(\mathrm{I}_{x, d})$ is satisfied, then $(\mathrm{I}_{x, p})$ is satisfied for any $p \geq d$.
\medskip

Now, if we assume that the criterion $(\mathrm{I}_{x, d})$ is satisfied for any point $x \in E \setminus F^{(2)}$, the metric $h_{\mathrm{Berg}}$ on $U^0 \Subset \Omega$ being equivalent to the standard metric $h_{\mathbb C^n}$. Thus, we see that the norm of any $d$-vector field in $\bigwedge^d T_W$ is bounded from above on $W$, for the metric $\bigwedge^d h_{\mathrm{Berg}}$. 

Moreover, the Bergman metric has negative holomorphic curvature, so it has Poincar\'e growth near any point $x \in D \setminus F^{(2)}$ by the argument already used in Lemma \ref{lembound} (see \cite{brucad17}). This implies that $h_{\mathrm{Berg}}$ is bounded on $T_{\widetilde{X}} (-\log D)$ near such any point of $D \setminus F^{(2)}$.

More precisely, we get the following result.

\begin{prop} \label{critborne} Assume that $(\mathrm{I}_{x, d})$ is satisfied for any $x \in E \setminus F^{(2)}$. Then $\bigwedge^d h_{\mathrm{Berg}}$ is locally bounded from above as a singular metric on $\bigwedge^d T_{\widetilde{X}} (- \log D)$.
\end{prop}

\begin{proof} It only remains to be shown that $\bigwedge^d h_{\mathrm{Berg}}$ is locally bounded as a metric on $\bigwedge^d T_{\widetilde{X}} (- \log D)$ near $F^{(2)}$. The metric $\bigwedge^d h_{\mathrm{Berg}}$ has negative curvature on $\widetilde{X} \setminus (E \cup D)$, so if $s$ is a local holomorphic section of $\bigwedge^d T_{\widetilde{X}} (- \log D)$, defined on some open subset $O$, the function $f = \log \norm{s}^2_{\bigwedge^d h_{\mathrm{Berg}}}$ is plurisubharmonic on $O \setminus E$. Since $\bigwedge^d h_{\mathrm{Berg}}$ is locally bounded near $(E + D) \setminus F^{(2)}$, $f$ is locally bounded at these points, so extends as a plurisubharmonic function across $(E + D) \setminus F^{(2)}$. Finally, $\mathrm{codim}\, F^{(2)} \geq 2$, so $f$ also extends across $E^{(2)}$ as a psh function, and thus is locally bounded from above on $O$.
\end{proof}

The previous considerations permit to give the following refinement of Theorem \ref{thmpartial} in this context. As before, we let $L_0 = \mathcal O_{\widetilde{X}} \left( \pi^\ast K_{\overline{X}^{BB}} \right)$.

\begin{thm} \label{thmsinghyp} Let $p \in \llbracket 1, n \rrbracket$. Suppose that there exists a rational number $\alpha > \frac{1}{C_p}$ such that $L_{\alpha} = L_0 \otimes \mathcal O ( - \alpha D)$ is effective, and that the condition $(\mathrm{I}_{x, p})$ holds for any $x \in E \setminus F^{(2)}$. Then the conclusion of Theorem \ref{thmpartial} holds, with $D \cup \mathbb B(L_\alpha)$ replaced by $D \cup \mathbb B(L_\alpha) \cup E$. 
\end{thm}
\begin{proof} Let us explain how to prove the first point of the conclusion, the second one being similar. Let $V \subset \widetilde{X}$ be a subvariety of dimension larger than $p$, and assume that $V \not\subset D \cup \mathbb B(L_\alpha) \cup E$. If $\widetilde{V}$ is a resolution of the singularities of $V$, we construct a metric on $K_{\widetilde{V}}$ of the form
\begin{equation} \label{eqdefihtilde}
\widetilde{h} = \norm{s}^{2 \beta} \mathrm{det} (h_{\mathrm{Berg}} |_{T_V}^\ast),
\end{equation}
for some section $s \in \Gamma \left( L_0^{r} \otimes \mathcal O (- q D) \right)$. The rest of the proof goes exactly in the same way ; note that Lemma \ref{loggrowth} extends to this more general context, because of our assumption that $\Gamma$ has no fixed component in codimension $1$. Indeed, the proof of \cite{crt17} gives the first part of the proof of Lemma \ref{loggrowth}. Also, we can use the same graph argument as in the second part of the proof, to carry the estimate on the birational model for $\overline{X}$ used in \cite{crt17}, to the model we used here.

The only missing element is to show that $\widetilde{h}^\ast$ is locally bounded near $E \cap V$. This part is provided by Proposition \ref{critborne}: this result implies indeed that the metric $ \mathrm{det} (h_{\mathrm{Berg}} |_{T_V}^\ast)$ is locally bounded on $\bigwedge^{\dim V} T_{\widetilde V \setminus D}$.
\end{proof}

It is not hard to formulate a version of Theorem \ref{thmsinghyp} free of the conditions $(\mathrm{I}_{x,d})$, by introducing the following quantities.

\begin{defi} \label{defibetaid} Let $p \in \llbracket 1, n \rrbracket$. For any component $E_i$ of $E$, let $\beta^i_p \geq 0$ be the smallest rational number such that for any $x \in E_i \setminus F^{(2)}$, if we let $(a_1, ..., a_n)$ and $r$ be as in Condition $(\mathrm{I}_{x, d})$, we have, for any choice of $p$ distinct elements $a_{I_1}, ..., a_{I_p}$, 
$$
\frac{\sum_{j = 1}^{p} a_{I_j}}{r} \geq 1 - \beta_p^i.
$$
\end{defi}

\begin{rem} Let $G_i$ be the isotropy group of the subvariety $V_i \subset \overline{X}$ on which the component $E_i$ projects. With the same notations as above, we necessarily have $r \leq \abs{G_i}$ since $\mathbb Z / r \mathbb Z$ embeds as a subgroup of $G_i$. Moreover, for each $l$ we have $a_l \geq 1$, since no codimension $1$ component is fixed. This implies that $\sum_{j = 1}^p a_{I_j} \geq p$, so
$$
\beta_p^i \leq \max \left( 0, 1 - \frac{p}{\abs{G_i}} \right)
$$
in general. In particular, if $p \geq \max_i \abs{G_i}$, then $\beta_p^i = 0$ for any $i$, and the condition $(\mathrm{I}_{x, p})$ is satisfied everywhere.
\end{rem}

Then, with the same proof as in Theorem \ref{thmsinghyp}, we can state the following result, which generalizes the criterions for hyperbolicity of \cite{crt17}:

\begin{corol} \label{corolsing} Let $\Delta_p = D + \sum_i \beta^i_p E_i$. Suppose that there exists a rational number $\alpha > \frac{1}{C_p}$ such that the $\mathbb Q$-line bundle $L_{\alpha, p} = \mathcal O_{\widetilde{X}} \left(L_0 \otimes  \mathcal O(- \alpha \Delta_p)  \right)$ is effective. Then the conclusion of Theorem \ref{thmpartial} holds, with $D \cup \mathbb B(L_\alpha)$ replaced by $D \cup E \cup \mathbb B(L_{\alpha, p})$.
\end{corol}
\begin{proof}
The proof of the first point goes in the same lines as for Theorem \ref{thmsinghyp}: we resume the same notations, \emph{mutatis mutandis}. In this case, the metric $\bigwedge^p h_{\mathrm{Berg}}$ is not necessarily bounded on $\Lambda^p T_{\widetilde{X}} (- \log D)$. However, we see easily that we can choose $s$ so that the metric $\widetilde{h}$ of \eqref{eqdefihtilde} is still locally bounded on $K_{\widetilde{V}}^\ast$. Indeed, for $m$ high enough, we can find a section $s$ of $L_0^{\otimes m}$ vanishing at some order $l_i$ along any $E_i$, so that $\frac{l_i}{m} \geq \beta^i_p$, and  $s|_V \neq 0$. Expending equation \eqref{eqaj} and using Definition \ref{defibetaid}, we see easily that $\widetilde{h}$ is locally bounded near $E$ and $D$ if we choose $\beta$ as in the proof of Theorem \ref{thmlincomb} (with $\lambda = 1$). As before, for such a choice of $\beta$, the metric $\widetilde{h}$ has positive curvature on its smooth locus. This permits to conclude that $K_{\widetilde{V}}$ is big using again \cite{bou02}.

The proof of the $p$-measure hyperbolicity can be extended in the same manner.
\end{proof}

We now get back to the study of $\overline{A_g}$. Let $\widetilde{A_g} \overset{\pi}{\longrightarrow} \overline{A_g}^{BB}$ be any resolution of singularities. Recall the following result of Weissauer \cite{weissauer86}:

\begin{prop}[{\cite[Zusammenfassung, p. 218]{weissauer86}}] \label{propcritag} The condition $(\mathrm{I}_{x, p})$ holds for any simple point of the exceptional divisor of $\widetilde{A_g}$, as soon as $p \geq \dim A_g - g + 7$.
\end{prop}

Thus, Theorem \ref{thmsinghyp} yields the following improved version of the main result of \cite{weissauer86}, which is equivalent to Corollary \ref{corolsubv}:

\begin{corol} Let $\widetilde{A_g}$ be any resolution of singularities of $\overline{A_g}^{BB}$. Then, any subvariety $V$ which is not included in the boundary nor in the exceptional divisor, and such that $\mathrm{codim} V \leq g - 12$, is of general type.
\end{corol}
\begin{proof}
To deduce the result from Theorem \ref{thmsinghyp}, we just need to confirm that if $\mathrm{codim} V = \frac{g(g+1)}{2} - p \geq g - 12$, then $L_{\alpha}$ has its base locus included in $D$ for some $\alpha > \frac{1}{C_p}$. By Proposition \ref{weissauer86}, this will be true if $\frac{g+1}{12} > \frac{1}{C_p}$. By the table 1, we have $C_p = \frac{g + 1 - \mathrm{codim} V}{g + 1}$ in that case, so the result follows.
\end{proof}

\section*{Annex. Python procedures for the computation of the constants $C_p$ for the domain $\mathbb H_g$}
\addcontentsline{toc}{section}{Annex}

In this annex, we describe the Python functions which were used in the proof of Proposition \ref{propconstants} and \ref{corolmgn}. We resume the notations of Section \ref{siegelcase}.

\medskip

For a given $(g,p)$, the function \verb?setGamma? returns the set of all pairs $(k, \Gamma_0)$, arising in the following manner. Let $\Gamma \in \mathcal S_p$ be such that the off-diagonal element of $\Gamma$ are maximal for the lexicographic order. If the number of diagonal elements is an integer $k \leq g - 1$, and if $\Gamma_0$ is the tuple $(a_1, ..., a_{g-1})$, where $a_i$ is the number of off-diagonal elements of the $i$-th row of $\Gamma$, then the pair $(k, \Gamma_0)$ is an acceptable pair. 

It proceeds inductively, constructing first the pairs $(k', \Gamma'_0)$ associated to $(g, p-1)$. Then, it determines all the possible pairs $(k, \Gamma_0)$, by adding 1 to the integers appearing in $(k', \Gamma'_0)$ whenever it is possible.

\begin{lstlisting}[frame=single]
def setGamma(g,p):
	if p == 1:
		res = set()
		res.add( (0, tuple([0]*(g-1))) )
		return res 
	l1 = list(setGamma(g, p-1)) 
	res = set()
	for t in l1:
		if t[0] < g - 1:
			res.add( (t[0] + 1, t[1]) )
		for i in range(g-2):
			if t[1][i] < g - 1 - i and (t[1][i] < t[1][i+1] 
					or t[1][i+1] == g - i - 2):
				nl = list(t[1])
				nl[i] = nl[i] + 1
				res.add((t[0], tuple(nl)))
		if t[1][g-2] == 0:
			nl = list(t[1])
			nl[g-2] = 1
			res.add((t[0], tuple(nl)))
	return res	
\end{lstlisting}	
\medskip

The function \verb?linPart? takes a pair $(k, \Gamma_0)$ provided by the function \verb?setGamma?. If $\Gamma \subset T$ is the associated subset, the function returns the list of elements $(c_i)_{1 \leq i \leq g}$, where $c_i$ is the number of times $i$ appears in an element of $\Gamma$, counting the diagonal elements twice. 

\begin{lstlisting}[frame=single]
def linPart(gamma):
	l = list(gamma[1])
	g = len(l) + 1
        res = [0]*(g)
	for i in range(g-1):
		res[i] = res[i] + l[i]
		for j in range(l[i]):
			res[g - j - 1] = res[g - j - 1] + 1
	for i in range(gamma[0]):
		res[g - i - 1] = res[g - i - 1] + 2
	return res	
\end{lstlisting}
\medskip

The function \verb?minQuad? takes a list of elements $(b_1, ..., b_g)$ such that $b_1 \leq ... \leq b_g$, and returns the minimum of the quadratic function $2 \sum_{i = 1}^{g} m_i^2 + \sum_{i=1}^{g} b_i m_i$, where $(m_i)_i \in \Delta_g$, using the formula of Lemma \ref{lemvanishing1}.

\begin{lstlisting}[frame=single]
def minQuad(b):
	t = len(b)
	while True: 
		S = sum( [ b[t-1] - b[i] for i in range(t) ] ) 
		if S < 4 or t == 1:
			break	
		t = t - 1
	S1 = sum([ b[i] for i in range(t) ])
	S2 = sum([ b[i]^2 for i in range(t) ])
	return (1/(8*t))*(S1 + 4)^2 - (1/8)*S2	
\end{lstlisting}
\medskip

The function \verb?C(g,p)? computes the value of the constant $C_p$ for the given value of $g$. To do this, it ranges among all the elements $(k, \Gamma_0)$ provided by the function \verb?setGamma?, and computes the value of $ \mu = \min_{\underline{m_0} \in \Delta_g} \mathcal F(\underline{m_0}, \Gamma)$, where $\Gamma$ is associated to $(k, \Gamma_0)$. According to Definition \ref{defiF}, if $k = g - 1$, then $\mu$ is equal to $2$ plus the number of elements on the first row of $\Gamma$, and if $k \leq g - 1$, the value of $\mu$ is computed using the function \verb?minQuad?.

\begin{lstlisting}[frame=single]
def C(g, p):
	s = list(setGamma(g,p))
	res = g + 1
	for gamma in s:
		if gamma[0] == g - 1:
			res = min(res, 2 + gamma[1][0])
		else:
			lin = linPart(gamma)
			res = min(res, minQuad(lin))
	return res
\end{lstlisting}
\medskip

The function \verb?C1? takes a pair $(g, p)$, and returns $(g+1)C_p$, where $C_p$ is computed by the table of Proposition \ref{propconstants}.

\begin{lstlisting}[frame=single]
def c(g, k, r):
	if g-k == 1:
		return r + 2
	if r == 0:
		return 2/(g-k)
	if g-k == 2 and r == 1:
		return 23/16
	if g-k == 2 and r == 2:
		return 7/4
	if g-k == 2 and r == 3:
		return 31/16
	if g-k == 3 and r == 1:
		return 11/12
	if g-k == 4 and r == 1:
		return 21/32
	return 2/(g-k-1)

def C1(g,p):
	k = RR( (sqrt(1 + 8*(p-1)) - 1)/2).floor()
	r = p - 1- k*(k+1)/2
	return c(g, k, r)
\end{lstlisting}

\bibliographystyle{amsalpha}
\bibliography{biblio}
\textsc{Beno\^it~Cadorel, Institut de Math\'{e}matiques de Toulouse, UMR 5219, Université de Toulouse; CNRS, UPS IMT, 118 route de Narbonne, F-31062, Toulouse Cedex 9, France} \par\nopagebreak
  \textit{E-mail address}: \texttt{benoit.cadorel@math.univ-toulouse.fr}
\vspace{0pt}
\end{document}